\documentclass[a4paper, 11pt]{amsart}

\usepackage[letterpaper,margin=1in]{geometry}
\usepackage{amsmath, amsthm, amsfonts, amssymb}
\usepackage{xcolor}
\usepackage{fancyhdr}
\usepackage{microtype} 
\usepackage[colorlinks = true]{hyperref}
\usepackage{paralist}
\hypersetup{
    colorlinks,
    linkcolor={red!50!black},
    citecolor={green!50!black},
    urlcolor={blue!80!black}
}

\newcommand{\bfj}{\mathbf{j}}

\newcommand{\calI}{\mathcal{I}}

\newcommand{\calL}{\mathcal{L}}

\newcommand{\calO}{\mathcal{O}}

\newcommand{\calS}{\mathcal{S}}
\newcommand{\calT}{\mathcal{T}}

\newcommand{\bbA}{\mathbb{A}}

\newcommand{\bbC}{\mathbb{C}}

\newcommand{\bbJ}{\mathbb{J}}
\newcommand{\bbK}{\mathbb{K}}

\newcommand{\bbN}{\mathbb{N}}

\newcommand{\bbP}{\mathbb{P}}

\newcommand{\bbX}{\mathbb{X}}
\newcommand{\bbY}{\mathbb{Y}}

\newcommand{\sfx}{\mathrm{x}}
\newcommand{\sfy}{\mathrm{y}}

\theoremstyle{theorem}
\newtheorem{theorem}{\sc Theorem}[section]  
\newtheorem{proposition}[theorem]{\sc Proposition}   
\newtheorem{corollary}[theorem]{\sc Corollary}        
\newtheorem{lemma}[theorem]{\sc Lemma}                
\newtheorem{definition}[theorem]{\sc Definition}
\newtheorem{example}[theorem]{\sc Example}
\newtheorem{question}{\sc Question}

\theoremstyle{remark}
\newtheorem*{notation}{{\it Notation}}
\newtheorem{remark}[theorem]{Remark}

\usepackage{stmaryrd}
\usepackage{tikz, subcaption}
\usepackage{xypic}
\usepackage{graphicx}
\usepackage{stackengine}

\numberwithin{equation}{section}

\newcommand{\tang}{\calT}
\newcommand{\Sym}{\mathit{Sym}}
\newcommand{\HF}{\mathrm{HF}}

\title[Tangential varieties of Segre-Veronese surfaces are never defective]{Tangential varieties of Segre-Veronese surfaces \\ are never defective}

\author[M. V. Catalisano]{Maria Virginia Catalisano}
\address[M. V. Catalisano]{Dipartimento di Ingegneria Meccanica, Energetica, Gestionale e dei Trasporti, Universita degli studi di Genova, Genoa, Italy}
\email{catalisano@diptem.unige.it}

\author[A. Oneto]{Alessandro Oneto}
\address[A. Oneto]{Universitat Politecnica de Catalunya, Department of Mathematics, Barcelona, Spain}
\email{alessandro.oneto@upc.edu}

\newcommand{\Alessandro}[1]{#1}

\begin{document}

\maketitle

\begin{abstract}
We compute the dimensions of all the secant varieties to the tangential varieties of all Segre-Veronese surfaces. We exploit the typical approach of computing the Hilbert function of special $0$-dimensional schemes on projective plane by using a new degeneration technique.
\end{abstract}

\section{Introduction}

The study of {\it secant varieties} and {\it tangential varieties} is very classical in Algebraic Geometry and goes back to the school of the XIX century. In the last decades, these topics received renewed attention because of their \Alessandro{connections} with more applied sciences \Alessandro{which uses {\it additive decompositions of tensors}}. \Alessandro{For us,} {\it tensors} are multidimensional arrays \Alessandro{of complex numbers and} classical geometric objects as Veronese, Segre, Segre-Veronese varieties, and their tangential varieties, are parametrised by tensors with particular symmetries and structure. \Alessandro{Their $s$-secant variety is the closure of the locus of linear combinations of $s$ many of those} particular tensors. We refer to \cite{L} for an exhaustive description of the fruitful use of classical algebraic geometry in problems regarding tensors decomposition. 

A very important invariant \Alessandro{of these varieties} is their {\it dimension}.

\medskip
A rich literature has been devoted to studying dimensions of secant varieties of special projective varieties. In particular, we mention: 
\begin{enumerate}
	\item[1.] {\it Veronese varieties}, completely solved by J. Alexander and A. Hirschowitz \cite{AH};
	\item[2.] {\it Segre varieties}, solved in few specific cases, e.g., see \cite{AOP, CGG02-Segre, CGG11-SegreP1};
	\item[3.] {\it Segre-Veronese varieties}, solved in even fewer specific cases, e.g., see \cite{AB13-SegreVeronese, CGG05-SegreVeronese, Abr08};
	\item[4.] {\it tangential varieties of Veronese varieties}, completely solved by H. Abo and N. Vannieuwenhoven, see \cite{AV18, BCGI09}.
\end{enumerate}

In this paper, we consider the following question.

\begin{question}\label{question: dimension secants}
	What is the dimension of secant varieties of tangential varieties of Segre-Veronese surfaces?
\end{question}

Let $a, b$ be positive integers. We define the {\bf Segre-Veronese embedding} of $\bbP^1\times \bbP^1$ in bi-degree $(a,b)$ as the embedding of $\bbP^1\times\bbP^1$ with the linear system $\calO_{\bbP^1\times\bbP^1}(a,b)$ of curves of bi-degree $(a,b)$, namely
\begin{eqnarray*}
\nu_{a,b} : &\bbP^1 \times \bbP^1 & \rightarrow \quad \quad \quad \bbP^{(a+1)(b+1)-1},\\
&([s_{0}:s_{1}],[t_{0}:t_{1}]) & \mapsto \quad [s_{0}^at_{0}^b:s_{0}^{a-1}s_{1}t_{0}^b:\ldots:s_{1}^at_{1}^b].
\end{eqnarray*}
 The image of $\nu_{a,b}$ is the {\bf Segre-Veronese surface} of bi-degree $(a,b)$, denoted by $SV_{a,b}$, \Alessandro{which is parametrized by {\it decomposable partially symmetric tensors}.}

Let $V_1$ and $V_2$ be two $2$-dimensional $\bbC$-vector spaces. Let $\Sym(V_i) = \bigoplus_{d\geq 0}\Sym^d(V_i)$ be the symmetric algebra over $V_i$, for $i = 1,2$. If we fix basis $(x_0,x_1)$ and $(y_0,y_1)$ for $V_1$ and $V_2$, respectively, then we have the identifications of the respective symmetric algebras with the polynomial rings \Alessandro{
\begin{equation}\label{eq: identif}
\begin{array}{r c l c r c l}
	\Sym(V_1) &\simeq& \bbC[x_0,x_1]  & \text{and}  & \Sym(V_2) & \simeq & \bbC[y_0,y_1], \\
	v_1 = (v_{10},v_{11}) &\leftrightarrow& \ell_1 = v_{10}x_0+v_{11}x_1 & & v_2 = (v_{20},v_{21}) &\leftrightarrow& \ell_2 = v_{20}y_0+v_{21}y_1  \\
\end{array}
\end{equation}
}
Therefore, $\Sym(V_1) \otimes \Sym(V_2)$ is identified with the bi-graded ring $\calS = \bbC[x_0,x_1;y_0,y_1] = \bigoplus_{a,b}\calS_{a,b}$, where $\calS_{a,b}$ is the $\bbC$-vector space of bi-homogeneous polynomials of bi-degree $(a,b)$, i.e., $\calS_{a,b} = \Sym^a(V_1) \otimes_{\bbK} \Sym^b(V_2)$. \Alessandro{In particular, we consider the monomial basis of $\calS_{a,b}$ given by 
$$
	\left\{{a \choose i_0}{b \choose j_0}x_0^{i_0}x_1^{i_1}y_0^{j_0}y_1^{j_1} ~|~ \substack{i_0,i_1,j_0,j_1 \geq 0 \\ i_0+i_1 = a, j_0+j_1 = b}\right\}.
$$
In this way, the Segre-Veronese embedding of $\bbP^1\times\bbP^1$ in bi-degree $(a,b)$ can be rewritten as }
\begin{eqnarray*}
\nu_{a,b} : & \bbP(V_1) \times \bbP(V_2) & \rightarrow \quad \bbP(\Alessandro{\Sym^a(V_1)\otimes\Sym^b(V_2)}),\\
&([v_1],[v_2]) & \mapsto [v_1^{\otimes a}\otimes v_2^{\otimes b}].
\end{eqnarray*}
Throughout all the paper\Alessandro{, by \eqref{eq: identif}, we} identify the tensor $v_1^{\otimes a}\otimes v_2^{\otimes b}$ with the polynomial $\ell_1^a\ell_2^b$ \Alessandro{and we view the Segre-Veronese variety of bi-degree $(a,b)$ as the projective variety parametrized by these particular bi-homogeneous polynomials.}

\smallskip
Given any projective variety $X \subset \bbP^N$, we define the {\bf tangential variety} of $X$ as the Zariski closure of the union of the tangent spaces at smooth points of $X$, i.e., if $U \subset X$ denotes the open subset of smooth points of $X$, then it is
$$
	\tang(X) := \overline{\bigcup_{P \in U} T_P(X)} \subset \bbP^N,
$$
where $T_P(X)$ denotes the tangent space of $X$ at the point $P$. 

\smallskip
Given any projective variety $X \subset \bbP^N$, we define the {\bf $s$-secant variety} of $X$ as the Zariski closure of the union of all linear spans of $s$-tuples of points on $X$, i.e.,
\begin{equation}\label{equation: expected dimension}
	\sigma_s(X) := \overline{\bigcup_{P_1,\ldots,P_s \in X} \left\langle P_1,\ldots,P_s \right\rangle} \subset \bbP^N,
\end{equation}
where $\langle - \rangle$ denotes the linear span of the points. 

As we said before, we are interested in the dimension of these varieties. By parameter count, we have an {\bf expected dimension} of $\sigma_s(X)$ which is
$$
	{\rm exp}.\dim\sigma_s(X) = \min\{N, s\dim(X) + (s-1)\}.
$$
However, we have varieties whose $s$-secant variety has dimension smaller than the expected one and we call them {\it defective varieties}. In this article we prove that \Alessandro{the} tangential varieties \Alessandro{of all the Segre-Veronese surfaces} are never defective. 

\setcounter{section}{4}
\setcounter{theorem}{5}
\begin{theorem}
	Let $a,b$ be positive integers with $ab > 1$. Then, the tangential variety of any Segre-Veronese surface $SV_{a,b}$ is non defective, i.e., all secant varieties have the expected dimension; namely,
	$$
		\dim\sigma_s(\tang(SV_{a,b})) = \min\{(a+1)(b+1), 5s\}-1.
	$$
\end{theorem}

\Alessandro{In order to prove our result, we use an approach already used in the literature which involves the study of {\it Hilbert functions} of {\it $0$-dimensional schemes} in the multiprojective space $\bbP^1\times\bbP^1$. In particular, first we use a method introduced by the first author, with A. V. Geramita and A. Gimigliano \cite{CGG05-SegreVeronese}, to reduce our computations to the standard projective plane; and second, we study the dimension of particular linear systems of curves with non-reduced base points by using degeneration techniques which go back to G. Castelnuovo, but have been refined by the enlightening work of J. Alexander and A. Hirschowitz \cite{AHb, AHa, AH, AH00}. However, as far as we know, the particular type of degeneration that we are using has not been exploited before in the literature and we believe that it might be useful to approach other similar problems.}

\setcounter{section}{1}

\paragraph*{\bf Structure of the paper.} In Section \ref{sec: secant and 0-dim schemes}, we show how the dimension of secant varieties can be computed by studying the Hilbert function of $0$-dimensional schemes and we introduce the main tools that we use in our proofs, such as the {\it multiprojective-affine-projective method} and {\it la m\'ethode d'Horace diff\'erentielle}. In Section \ref{sec: lemmata}, we consider the cases of small bi-degrees, i.e., when $b \leq 2$. These will be the base steps for our inductive proof of the general case that we present in Section~\ref{sec: main}.

\smallskip
\paragraph*{\bf Acknowledgements.}
The first author was supported by  the Universit\`a degli Studi di Genova through the \Alessandro{``FRA (Fondi per la Ricerca di Ateneo) 2015”}. The second author acknowledges ﬁnancial support from the Spanish Ministry of Economy and Competitiveness, through the María de Maeztu Programme for Units of Excellence in R\&D (MDM-2014-0445).

\section{Secant varieties and $0$-dimensional schemes}\label{sec: secant and 0-dim schemes}

In this section we recall some basic constructions and we explain how they are used to reduce the problem of computing dimensions of secant varieties to the problem of computing Hilbert functions of special $0$-dimensional schemes.

\subsection{Terracini's Lemma} A standard way to compute the dimension of an algebraic variety is to look at its tangent space at a general point. In the case of secant varieties, the structure of tangent spaces is very nicely described by a classical result of A.~Terracini \cite{T11}.

\begin{lemma}[Terracini's Lemma \cite{T11}]
	Let $X$ be a projective variety. Let $P_1,\ldots,P_s$ be general points on $X$ and let $P$ be a general point on their linear span. Then,
	$$
		T_P\sigma_s(X) = \left\langle T_{P_1}X,\ldots,T_{P_s}X\right\rangle.
	$$
\end{lemma}
\noindent \Alessandro{Therefore, in order to understand the dimension of the general tangent space of the secant variety $\sigma_s(X)$, we need to compute the dimension of the linear span of the tangent spaces to $X$ at $s$ general points. We do it in details for the tangential varieties of Segre-Veronese surfaces.}

\smallskip
Let $\ell_1 \in \calS_{1,0}$ and $\ell_2 \in \calS_{0,1}$. Then, we consider the bi-homogeneous polynomial $\ell_1^a\ell_2^b \in \calS_{a,b}$ which represents a point on the Segre-Veronese variety $SV_{a,b}$. Now, if we consider two general linear forms $m_1 \in \calS_{1,0}$ and $m_2 \in \calS_{0,1}$, then
$$
	\left.\frac{d}{dt}\right|_{t=0}(\ell_1+tm_1)^a(\ell_2+tm_2)^b = a\ell_1^{a-1}\ell_2^bm_1 + b\ell_1^a\ell_2^{b-1}m_2;
$$
therefore, we obtain that 
$$
	T_{[\ell_1^a\ell_2^b]} SV_{a,b} = \bbP\left( \left\langle \ell_1^{a-1}\ell_2^b \cdot\calS_{1,0},~\ell_1^{a}\ell_2^{b-1} \cdot \calS_{0,1} \right\rangle \right).
$$
Hence, the tangential variety $\tang(SV_{a,b})$ is the image of the embedding
\begin{eqnarray*}
\tau_{a,b} : & \big(\bbP(V_1)\times\bbP(V_2)\big)\times \big(\bbP(V_1)\times\bbP(V_2)\big)  & \rightarrow  ~~~~~~~~\bbP({\calS_{a,b}}),\\
&~~~\big(([\ell_1],[\ell_2])~;~([m_1],[m_2])\big) & \mapsto [\ell_1^{a-1}\ell_2^bm_1+\ell_1^a\ell_2^{b-1}m_2].
\end{eqnarray*}
\begin{remark}\label{rmk: (a,b) = (1,1)}
	The variety $SV_{1,1}$ is the Segre surface of $\bbP^3$ whose tangential variety clearly fills the entire ambient space. For this reason, we will always consider pairs of positive integers $(a,b)$ where at least one is strictly bigger than $1$. Hence, from now on, we assume $ab > 1$.
\end{remark}

Now, fix $\ell_1,m_1 \in \calS_{1,0}$ and $\ell_2,m_2 \in \calS_{0,1}$. For any $h_1,k_1 \in \calS_{1,0}$ and $h_2,k_2\in \calS_{0,1}$, we have
\begin{align*}
	\left.\frac{d}{dt}\right|_{t=0} (\ell_1 + th_1)^{a-1}&(m_1+tk_1)(\ell_2+th_2)^b + (\ell_1 + th_1)^{a}(\ell_2+th_2)^{b-1}(m_2+tk_2) =  \\
	& = (a-1)\ell_1^{a-2}h_1m_1\ell_2^b + \ell_1^{a-1}k_1\ell_2^b + b\ell_1^{a-1}m_1\ell_2^{b-1}h_2 +  \\ & \quad \quad + a\ell_1^{a-1}h_1\ell_2^{b-1}m_2 + (b-1)\ell_1^a\ell_2^{b-2}h_2m_2 + \ell_1^a\ell_2^{b-1}k_2.
\end{align*}
\Alessandro{Note that, if $b = 1$ (or $a = 1$, resp.), the summand where $\ell_2$ is appearing with exponent $(b-2)$ (or where $\ell_1$ is appearing with exponent $(a-2)$, resp.) vanishes since it appears multiplied by the coefficient $(b-1)$ (or $(a-1)$, resp.).}

Therefore, we have that, if $P = \tau_{a,b}\left(([\ell_1],[\ell_2]),([m_1,m_2])\right)\Alessandro{\in\calT(SV_{a,b})}$, then
\begin{align*}
	T_{P} (\tang(SV_{a,b})) =  \bbP\left( \left\langle \ell_1^a\ell_2^{b-1}\cdot \calS_{0,1},~ \ell_1^{a-1}\ell_2^b \cdot \calS_{1,0},~ \right.\right. &\left.\left. \ell_1^{a-2}\ell_2^{b-1}\big((a-1) m_1\ell_2 + a\ell_1 m_2 \big) \cdot \calS_{1,0}, \right.\right.  \\
	& \quad\quad\left.\left. \ell_1^{a-1}\ell_2^{b-2}\big(b m_1\ell_2 + (b-1)\ell_1 m_2 \big)\cdot \calS_{0,1} \right\rangle \right)
\end{align*}
From this description of the general tangent space to the tangential variety, we can conclude that the tangential variety $\tang(SV_{a,b})$ has the expected dimension.
\begin{lemma}\label{lemma: dimension tangential}
	Let $(a,b)$ be a pair of positive integers with $ab > 1$. Then, $\tang(SV_{a,b})$ is $4$-dimensional.
\end{lemma}
\begin{proof}
	Let $P$ be a general point of $\tang(SV_{a,b})$. Now, up to change of coordinates, we may assume
	$$
		\ell_1 = x_0,~m_1 = x_1 \quad \text{ and }  \quad \ell_2 = y_0, m_2 = y_1.
	$$
	By direct computations, we have that the affine cone over $T_P(\tang(SV_{a,b}))$ is
	\begin{align*}
		\left\langle x_0^ay_0^b,~ x_0^ay_0^{b-1}y_1,~ x_0^{a-1}x_1y_0^b,~ \right.& (a-1)x_0^{a-2}x_1^2y_0^b + a x_0^{a-1}x_1y_0^{b-1}y_1, \\ & \quad \quad \left.
		bx_0^{a-1}x_1y_0^{b-1}y_1 + (b-1)x_0^ay_0^{b-2}y_1^2\right\rangle.
	\end{align*}
	\Alessandro{As mentioned above, if either $a = 1$ (or $b = 1$, resp.), the summand where $x_0$ is appearing with exponent $(a-2)$ (or $y_0$ is appearing with exponent $(b-2)$, resp.) vanishes because it is multiplied by $(a-1)$ (or $(b-1)$, resp.). Since $ab > 1$, this} linear space clearly has dimension $5$. Thus, the claim follows.
\end{proof}
Therefore, by \eqref{equation: expected dimension}, we have that 
\begin{equation}\label{equation: expected dimension tangent}
	{\rm exp}.\dim \sigma_s(\tang(SV_{a,b})) = \min\left\{(a+1)(b+1), 5s\right\} - 1.
\end{equation}

\subsection{Apolarity Theory \Alessandro{and Fat points}}\label{section: apolarity}
For higher secant varieties, similar computations as Lemma \ref{lemma: dimension tangential} are not feasible. In order to overcome this difficulty, a classical strategy is to use {\it Apolarity Theory}. Here we recall the basic constructions, but, for exhaustive references on this topic, we refer to \cite{IK, G}.

\smallskip
Let $\calS = \bbC[x_0,x_1;y_0,y_1] = \bigoplus_{(a,b)\in\bbN^2} \calS_{a,b}$ be the bi-graded polynomial ring. Any bi-homogeneous ideal $I$ inherits the grading, namely, we have $I = \bigoplus_{a,b} I_{a,b}$, where $I_{a,b} = I \cap \calS_{a,b}$. Fixed a bi-degree $(a,b)$, for any $\alpha = (\alpha_0,\alpha_1)\in \bbN^2$ and $\beta = (\beta_0,\beta_1) \in \bbN^2$ such that $|\alpha| = \alpha_0+\alpha_1 = a$ and $|\beta| = \beta_0+\beta_1 = b$, we denote by $\sfx^\alpha\sfy^\beta$ the monomial $x_0^{\alpha_0}x_1^{\alpha_1}y_0^{\beta_0}y_1^{\beta_1}$; hence, for any polynomial $f \in \calS_{a,b}$, we write $f = \sum_{\substack{\alpha,\beta \in \bbN^2, \\ |\alpha|=a,~|\beta|=b}} f_{\alpha,\beta}\sfx^\alpha\sfy^\beta$. For any bi-degree $(a,b)\in\bbN^2$, we consider the non-degenerate apolar pairing 
$$
 	\circ : \calS_{a,b} \times \calS_{a,b} \longrightarrow \bbC, \quad\quad (f,g) \mapsto \sum_{\substack{\alpha,\beta \in \bbN^2, \\ |\alpha|=a,~|\beta|=b}} f_{\alpha,\beta}g_{\alpha,\beta}.
$$ 
Now, given a subspace $W \subset \calS_{a,b}$, we denote by $W^\perp$ the perpendicular space with respect to the apolar pairing, i.e, $W^\perp = \{f \in \calS_{a,b} ~|~ f \circ g = 0,~\forall g \in W\}$. From this definition, it is easy to prove that, given $W_1,\ldots,W_s \subset \calS_{a,b}$, we have
\begin{equation}\label{equation: intersection perp}
	\left\langle W_1,\ldots,W_s \right\rangle^\perp = W_1^\perp \cap \ldots \cap W_s^\perp.
\end{equation}
\begin{remark}\label{lemma: tangential 0-dimensional ideal}
	Let $P = \tau_{a,b}\big(([\ell_1],[\ell_2]),([m_1],[m_2])\big)\in \tang(SV_{a,b})$ be a general point and let $W$ be the affine cone over the tangent space, i.e., $\bbP(W) = T_P(\tang(SV_{a,b}))$. Then, we may observe that \Alessandro{$\wp^3_{a,b} \subset W^\perp \subset \wp^2_{a,b}$}, where $\wp$ is the ideal defining the point $([\ell_1],[\ell_2]) \in \bbP^1\times\bbP^1$.
	
Indeed, up to change of coordinates, we may assume
	$$
		\ell_1 = x_0,~m_1 = x_1 \quad \text{ and }  \quad \ell_2 = y_0, ~m_2 = y_1.
	$$
	Therefore, \Alessandro{as in the proof} of Lemma \ref{lemma: dimension tangential}, we have
	\begin{align}
		W = \left\langle x_0^ay_0^b,~ x_0^ay_0^{b-1}y_1,~ x_0^{a-1}x_1y_0^b,~ \right.& (a-1)x_0^{a-2}x_1^2y_0^{b} + a x_0^{a-1}x_1y_0^{b-1}y_1, \nonumber \\ & \quad \quad \left. bx_0^{a-1}x_1y_0^{b-1}y_1   +  (b-1)x_0^{a}y_0^{b-2}y_1^2\right\rangle \subset \calS_{a,b}. \label{eq: explicit tangent}
	\end{align}
	\Alessandro{It is easy to check that 
	 $$[\wp^2_{a,b}]^\perp = \left\langle x_0^ay_0^b,~ x_0^ay_0^{b-1}y_1,~ x_0^{a-1}x_1y_0^b \right\rangle \subset W \quad \text{and} \quad\wp^3_{a,b} = (x_1^3,x_1^2y_1,x_1y_1^2,y_1^3)_{a,b} \subset W^\perp;$$ therefore, since the apolarity pairing is non-degenerate, we obtain that
	$
		\wp^{3}_{a,b} \subset W^\perp \subset \wp^{2}_{a,b}.
	$}
\end{remark}

\begin{definition}
	Let $P \in \bbP^1\times\bbP^1$ be a point defined by the ideal $\wp$. We call {\bf fat point} of {\bf multiplicity $m$}, or $m$-fat point, and {\bf support in $P$}, the $0$-dimensional scheme $mP$  defined by the ideal $\wp^m$. 
\end{definition}

Therefore, from Remark \ref{lemma: tangential 0-dimensional ideal}, we have that the general tangent space to the tangential variety of the Segre-Veronese surface $SV_{a,b}$ is the projectivisation of a $5$-dimensional vector space whose perpendicular is the bi-homogeneous part in bi-degree $(a,b)$ of an ideal describing a $0$-dimensional scheme of \Alessandro{length} $5$ which is contained in between a $2$-fat point and a $3$-fat point. In the next lemma, we describe better the structure of the latter $0$-dimensional scheme.

\begin{lemma}\label{lemma: explicit tangent}
	Let $P = \tau_{a,b}\big(([\ell_1],[\ell_2]),([m_1],[m_2])\big)\in \tang(SV_{a,b})$ be a general point and let $W$ be the affine cone over the tangent space, i.e., $\bbP(W) = T_P(\tang(SV_{a,b}))$. Then, $W^\perp = [\wp^3 + I^2]_{a,b}$, where $\wp$ is the ideal of the point $P$ and $I$ is the principal ideal $I = (\ell_1m_2+\ell_2m_1)$.
\end{lemma}
\begin{proof}
		Up to change of coordinates, we may assume
	$$
		\ell_1 = x_0,~m_1 = x_1 \quad \text{ and }  \quad \ell_2 = y_0,~ m_2 = y_1.
	$$
	Let $\wp = (x_1,y_1)$ and $I = (x_1y_0+x_0y_1)$. \Alessandro{Now, from \eqref{eq: explicit tangent}, it follows} that
	$$
		W \subset \Alessandro{[\wp^3_{a,b}]^\perp\cap [I^2_{a,b}]^\perp \subsetneq [\wp^3_{a,b}]^\perp.}
	$$
	The latter inequality is strict; therefore, since \Alessandro{$\dim [\wp^3_{a,b}]^\perp = 6$, we get that $\dim [\wp^3_{a,b}]^\perp\cap [I^2_{a,b}]^\perp  \leq 5$} and, since  $\dim W = 5$, equality follows. Hence, since the apolarity pairing is non-degenerate,
	$$
		W^\perp = [\wp^3+I^2]_{a,b}.
	$$
\end{proof}
From these results, we obtain that the general tangent space to the tangential variety to the Segre-Veronese surface $SV_{a,b}$ can be described in terms of a connected $0$-dimensional scheme of \Alessandro{length} $5$ contained in between a \Alessandro{$2$-fat and a $3$-fat} point. Moreover, such a $0$-dimensional scheme is independent from the choice of $a,b$. Therefore, we introduce the following definition.
\begin{definition}\label{def: (3,2)-points P1xP1}
	Let $\wp$ be the prime ideal defining a point $P$ in $\bbP^1\times\bbP^1$ and let $(\ell)$ the principal ideal generated by a $(1,1)$-form passing through $P$. The $0$-dimensional scheme defined by $\wp^3+(\ell)^2$ is called {\bf $(3,2)$-fat point} with support at $P$. We call $P$ the {\bf support} and $\ell$ the {\bf direction} of the scheme.
\end{definition}
Now, by using Terracini's Lemma, we get that, from Lemma \ref{lemma: explicit tangent} and \eqref{equation: intersection perp}, 
\begin{equation}\label{equation: dimension of secants and HF}
	\dim \sigma_s(\tang(SV_{a,b})) = \dim \calS_{a,b}/I(\bbX)_{a,b}-1,
\end{equation}
where $\bbX$ is the union of $s$ many $(3,2)$-fat points with generic support. 

\medskip
Let us recall the definition of {\it Hilbert function}. Since we will use it both in the standard graded and in the multi-graded case, we present this definition in a general setting.
\begin{definition}
	Let $\calS = \bigoplus_{h\in H}\calS_h$ be a polynomial ring graded over a semigroup $H$. Let $I = \bigoplus_{h \in H} \subset \calS$ be a $H$-homogeneous ideal, i.e., an ideal generated by homogeneous elements. For any $h\in H$, we call {\bf Hilbert function} of $\calS/I$ in degree $h$, the dimension, as vector space, of the homogeneous part of degree $h$ of the quotient ring $\calS/I$, i.e.,
	$$
		\HF_{\calS/I}(h) = \dim [\calS/I]_{h} = \calS_{h}/I_{h}.
	$$
Given a $0$-dimensional scheme $\bbX$ defined by the ideal $I(\bbX) \subset \calS$, we say {\bf Hilbert function} of $\bbX$ when we refer to the Hilbert function of $\calS/I(\bbX)$.
\end{definition}
In the standard graded case we have $H = \bbN$, while in the bi-graded cases we have $H = \bbN^2$.

\medskip
Therefore, by \eqref{equation: dimension of secants and HF}, we reduced the problem of computing the dimension of secant varieties to the problem of computing the Hilbert function of a special $0$-dimensional scheme.

\begin{question}\label{question: 0-dim P1xP1}
{\it Let $\bbX$ be a union of $(3,2)$-fat points with generic support and generic direction.

\centerline{
For any $(a,b) \in \bbN^2$, what is the Hilbert function of $\bbX$ in bi-degree $(a,b)$?
}}
\end{question}

\smallskip
Also for this question we have an {\it expected} answer. Since $\bbX$ is a $0$-dimensional scheme in $\bbP^1\times\bbP^1$, if we represent the multi-graded Hilbert function of $\bbX$ as an infinite matrix $(\HF_\bbX(a,b))_{a,b\geq 0}$, then it is well-known that, in each row and column, it is strictly increasing until it reaches the degree of $\bbX$ and then it remains constant. Hence, if we let the support of $\bbX$ to be generic, we expect the Hilbert function of $\bbX$ to be the largest possible. Since a $(3,2)$-fat point has degree $5$, if $\bbX$ is a union of $(3,2)$-fat points with generic support, then
$$
	{\rm exp}.\HF_{\bbX}(a,b) = \min\left\{ (a+1)(b+1), 5s \right\}.
$$
As we already explained in \eqref{equation: dimension of secants and HF}, this corresponds to the expected dimension \eqref{equation: expected dimension tangent} of the $s$-th secant variety of the tangential varieties to the Segre-Veronese surfaces $SV_{a,b}$. 

\subsection{\Alessandro{Multiprojective-affine-projective method}}
In \cite{CGG05-SegreVeronese}, the authors defined a very powerful method to study Hilbert functions of $0$-dimensional schemes in multiprojective spaces. The method reduces those computations to the study of the Hilbert function of schemes in standard projective spaces, which might have higher dimensional connected components, depending on the dimensions of the projective spaces defining the multiprojective space. However, in the case of products of $\bbP^1$'s, we still have $0$-dimensional schemes in standard projective space, as we explain in the following.

\medskip
We consider the birational function
\begin{center}
\begin{tabular}{c c c c c c}
		$\phi :$ & $ \bbP^1 \times \bbP^1$ & $\dashrightarrow$ & $\bbA^2$ &  $\rightarrow$ & $\bbP^2$ \\
		& $([s_0:s_1],[t_0,t_1])$ &  $\mapsto$ & $(\frac{s_1}{s_0},\frac{t_1}{t_0})$ & $\mapsto$ & $[1:\frac{s_1}{s_0}:\frac{t_1}{t_0}] = [s_0t_0:s_1t_0:s_0t_1].$
\end{tabular}
\end{center}
\begin{lemma}{\rm \cite[Theorem 1.5]{CGG05-SegreVeronese}}
	Let $\bbX$ be a $0$-dimensional scheme in $\bbP^1\times \bbP^1$ with generic support, i.e., assume that the function $\phi$ is well-defined over $\bbX$. Let $Q_1 = [0:1:0], Q_2 = [0:0:1] \in \bbP^2$. Then,
	$$
		\HF_{\bbX}(a,b) = \HF_{\bbY}(a+b),
	$$
	where $\bbY = \phi(\bbX) + aQ_1 + bQ_2$.
\end{lemma}
Therefore, in order to rephrase Question \ref{question: 0-dim P1xP1} as a question about the Hilbert function of $0$-dimensional schemes in standard projective spaces, we need to understand what is the image of a $(3,2)$-fat point of $\bbP^1\times\bbP^1$ by the map $\phi$.

\smallskip
Let $z_0,z_1,z_2$ be the coordinates of $\bbP^2$. Then, the map $\phi$ corresponds to the function of rings
\begin{center}
\begin{tabular}{c c c c}
	$\Phi :$ & $\bbC[z_0,z_1,z_2]$ & $\rightarrow$ & $\bbC[x_0,x_1;y_0,y_1]$; \\
	& $z_0$ & $\mapsto$ & $x_0y_0$, \\	
	& $z_1$ & $\mapsto$ & $x_1y_0$, \\
	& $z_2$ & $\mapsto$ & $x_0y_1$. \\
\end{tabular}
\end{center}
By genericity, we may assume that the $(3,2)$-fat point $J$ has support at $P = ([1:0],[1:0])$ and it is defined by $I(J) = (x_1,y_1)^3 + (x_0y_1+x_1y_0)^2$. By construction, we have that $I(\phi(J)) = \Phi(I(J))$ and it is easy to check that
\begin{equation}\label{eq: (3,2)-points construction}
	\Phi(I(J)) = (z_1,z_2)^3 + (z_1+z_2)^2.
\end{equation}
Therefore, $\phi(J)$ is a $0$-dimensional scheme obtained by the scheme theoretic intersection of a triple point and a double line passing though it. In the literature also these $0$-dimensional schemes are called {\it $(3,2)$-points}; e.g., see \cite{BCGI09}. We call {\bf direction} the line defining the scheme. This motivates our Definition \ref{def: (3,2)-points P1xP1} which is also a slight abuse of the name, but we believe that it will not rise any confusion in the reader since the ambient space will always be clear in the exposition. We consider a generalization of this definition in the following section. By using these constructions, Question \ref{question: 0-dim P1xP1} is rephrased as follows. 

\begin{question}\label{question: 0-dim P2}
{\it Let $\bbY$ be a union of $s$ many $(3,2)$-points with generic support and generic direction in $\bbP^2$. 

\centerline{
For any $a,b$, let $Q_1$ and $Q_2$ be generic points and consider $\bbX = \bbY + aQ_1 + bQ_2$.}

\centerline{
What is the Hilbert function of $\bbX$ in degree $a+b$?
}}
\end{question}

\begin{notation}
	Given a $0$-dimensional scheme $\bbX$ in $\bbP^2$, we denote by $\calL_d(\bbX)$ the linear system of plane curves of degree $d$ having $\bbX$ in the base locus, i.e., the linear system of plane curves whose defining degree $d$ equation is in the ideal of $\bbX$. Similarly, if $\bbX'$ is a $0$-dimensional scheme in $\bbP^1\times\bbP^1$, we denote by $\calL_{a,b}(\bbX')$ the linear system of curves of bi-degree $(a,b)$ on $\bbP^1\times\bbP^1$ having $\bbX'$ in the base locus. Therefore, Question \ref{question: 0-dim P1xP1} and Question \ref{question: 0-dim P2} are equivalent of asking the dimension of these types of linear systems of curves. We define the {\bf virtual dimension} as
	$$
		{\it vir}.\dim\calL_d(\bbX) = {d+2 \choose 2} - \deg(\bbX) \quad \text{ and } \quad 	{\it vir}.\dim\calL_{a,b}(\bbX') = (a+1)(b+1) - \deg(\bbX').
	$$
	Therefore, the {\bf expected dimension} is the maximum between $0$ and the virtual dimension. We say that a $0$-dimensional scheme $\bbX$ in $\bbP^2$ ($\bbX'$ in $\bbP^1\times\bbP^1$) imposes {\bf independent conditions} on $\calO_{\bbP^2}(d)$ (on $\calO_{\bbP^1\times\bbP^1}(a,b)$, respectively) if the dimension of $\calL_{d}(\bbX)$ ($\calL_{a,b}(\bbX')$, respectively) is equal to the virtual dimension.
\end{notation}

\subsection{M\'ethode d'Horace diff\'erentielle}
From now on, we focus on Question \ref{question: 0-dim P2}. We use a degeneration method, known as {\it differential Horace method}, which has been introduced by J. Alexander and A. Hirschowitz, by extending a classical idea which was already present in the work of G. Castelnuovo. They introduced this method in order to completely solve the problem of computing the Hilbert function of a union of $2$-fat points with generic support in $\bbP^n$ \cite{AHb,AHa,AH}.

\begin{definition}
	In the algebra of formal functions $\bbC\llbracket z_1,z_2\rrbracket$, we say that an ideal is {\bf vertically graded} with respect to $z_2$ if it is of the form
	$$
		I = I_0 \oplus I_1z_2 \oplus I_2z_2^2 \oplus \ldots \oplus I_mz_2^m \oplus (z_2^{m+1}),\quad \text{where $I_i$'s are ideals in $\bbC \llbracket z_1\rrbracket$.}
	$$	
	If $\bbX$ is a connected $0$-dimensional scheme in $\bbP^2$ and $C$ is a curve through the support $P$ of $\bbX$, we say that $\bbX$ is {\bf vertically graded} with {\bf base} $C$ if there exist a regular system of parameters $(z_1,z_2)$ at $P$ such that $z_2 = 0$ is the local equation of $C$ and the ideal of $\bbX$ in $\calO_{\bbP^2,P} \simeq \bbC \llbracket z_1,z_2 \rrbracket$ is vertically graded.
\end{definition}

Let $\bbX$ be a vertically graded $0$-dimensional scheme in $\bbP^2$ with base $C$ and let $j\geq 1$ be a fixed integer; then, we define:
\begin{align*}
	j{\bf -th ~Residue}: & \quad {\rm Res}_C^j(\bbX), \text{ the scheme in $\bbP^2$ defined by } \calI_{\bbX} + (\calI_\bbX : \calI^{j}_C) \calI_C^{j-1}. \\
	j{\bf -th ~Trace}: & \quad {\rm Tr}_C^j(\bbX), \text{ the subscheme of $C$ defined by }  (\calI_\bbX : \calI^{j-1}_C) \otimes \calO_{C}.
\end{align*}
Roughly speaking, we have that, in the $j$-th residue, we remove the $j$-th {\it slice} of the scheme $\bbX$; while, in the $j$-th trace, we consider only the $j$-th {\it slice} as a subscheme of the curve $C$.

In the following example, we can see how we see as vertically graded schemes the $(3,2)$-fat points we have introduced before.

\begin{example}\label{example: (3,2)-point}
	Up to a linear change of coordinates, we may assume that the scheme $J$ constructed in \ref{eq: (3,2)-points construction} is defined by the ideal
	$
		I(J) = (z_1,z_2)^3 + (z_1)^2 = (z_1^2,z_1z_2^2,z_2^3).
	$
	Therefore, we have that, in the local system of parameters $(z_1,z_2)$, the scheme $J$ is vertically graded with respect to the $z_2$-axis defined by $\{z_1 = 0\}$; indeed, we have the two vertical layers given by
	$$
		I(X) = I_0 \oplus I_1z_1 \oplus (z_1)^2, \text{ where $I_0 = (z_2^3)$ and $I_1=(z_2^2)$};
	$$
	at the same time, we have that it is also vertically graded with respect to the $z_1$-axis defined by $\{z_2 = 0\}$; indeed, we have the three horizontal layers given by
	$$
		I(J) = I_0 \oplus I_1z_2 \oplus I_2z_2^2 \oplus (z_2)^3, \text{ where $I_0 = I_1 = (z_1^2)$ and $I_2=(z_1)$}.
	$$
	We can visualize $J$ as in Figure \ref{figure: (3,2)-point}, where the black dots correspond to the generators of the $5$-dimensional vector space $\bbC\llbracket z_1,z_2 \rrbracket / I(J) = \left\langle 1, z_1, z_2, z_1^2, z_1z_2 \right\rangle$.
\begin{figure}[h]
\begin{center}
\begin{tikzpicture}[line cap=round,line join=round,x=1.0cm,y=1.0cm,scale = 0.7]
\clip(-2.,-1.) rectangle (3.,3.);
\draw [line width=1.pt,dash pattern=on 1pt off 2pt] (0.,-1.) -- (0.,3.);
\draw [color = black] (-0.5,2.7) node {$z_2$};
\draw [line width=1.pt,dash pattern=on 1pt off 2pt,domain=-2.:3.] plot(\x,{(-0.-0.*\x)/-1.});
\draw [color = black] (2.5,-0.5) node {$z_1$};
\begin{scriptsize}
\draw [fill=black] (1.,0.) circle (3.5pt);
\draw [fill=black] (0.,1.) circle (3.5pt);
\draw [fill=black] (1.,1.) circle (3.5pt);
\draw [fill=black] (0.,2.) circle (3.5pt);
\draw[color=black] (0.22,8.17) node {$f$};
\draw[color=black] (-10.12,0.33) node {$g$};
\draw [fill=black] (0.,0.) circle (3.5pt);
\end{scriptsize}
\end{tikzpicture}
\end{center}
\caption{A representation of the $(3,2)$-point $J$, defined by the ideal $I(J) = (z_1,z_2)^3 + (z_1)^2$, as a vertically graded scheme. }
\label{figure: (3,2)-point}
\end{figure}
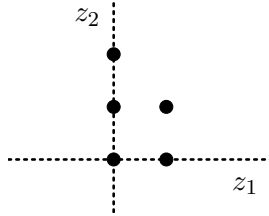

Therefore, if we consider $J$ as vertically graded scheme with base the $z_2$-axis, we compute the $j$-th residue and trace, for $j = 1,2$, as follows:

\begin{minipage}{0.6\textwidth}
\begin{align*}
	{\rm Res}^1_{z_2}(X) : & \quad (z_1^2,z_1z_2^2,z_2^3) + (z_1^2,z_1z_2^2,z_2^3) : (z_1) = (z_1,z_2^2); & \\
	{\rm Tr}^1_{z_2}(X) : & \quad (z_1^2,z_1z_2^2,z_2^3) \otimes \bbC\llbracket z_1,z_2 \rrbracket / (z_1) = (z_2^3); \\
	~ \\
	{\rm Res}^2_{z_2}(X) : & \quad (z_1^2,z_1z_2^2,z_2^3) + ((z_1^2,z_1z_2^2,z_2^3) : (z_1^2))\cdot (z_1) = (z_1,z_2^3); & \\
	{\rm Tr}^2_{z_2}(X) : & \quad \left((z_1^2,z_1z_2^2,z_2^3) : (z_1)\right)\otimes \bbC\llbracket z_1,z_2 \rrbracket / (z_1) = (z_2^2).
\end{align*}
\end{minipage}
\begin{minipage}{0.35\textwidth}
\begin{tikzpicture}[line cap=round,line join=round,x=1.0cm,y=1.0cm,scale = 0.6]
\clip(-0.5,-0.5) rectangle (7.,5.5);
\draw [line width=1.pt,dash pattern=on 1pt off 2pt] (1.,0.5) -- (1.,5.5);
\draw [line width=1.pt,dash pattern=on 1pt off 2pt] (5.,0.5) -- (5.,5.5);
\begin{scriptsize}
\draw [fill=black] (1.,4.) circle (4.0pt);
\draw [fill=black] (1.,3.) circle (4.0pt);
\draw [fill=black] (1.,2.) circle (4.0pt);
\draw [color=black] (2.,3.) circle (3.5pt);
\draw [color=black] (2.,2.) circle (3.5pt);
\draw [fill=black] (5.,3.) circle (4.0pt);
\draw [fill=black] (5.,2.) circle (4.0pt);
\draw [color=black] (5.98,2.99) circle (3.5pt);
\draw [color=black] (5.98,1.99) circle (3.5pt);
\draw[color=black] (5.166432702358458,6.933351817483886) node {$g_1$};
\draw [color=black] (6.,4.) circle (3.5pt);
\end{scriptsize}
\node[draw] at (1.5,0.05) {$j = 1$};
\node[draw] at (5.5,0.05) {$j = 2$};
\end{tikzpicture}
\end{minipage}

\bigskip
Similarly, if we consider it as vertically graded with respect to the $z_1$-axis, we compute the $j$-th residue and trace, for $j = 1,2,3$, as follows:

\begin{minipage}{0.55\textwidth}
\begin{align*}
	{\rm Res}^1_{z_1}(X) : & \quad (z_1^2,z_1z_2^2,z_2^3) + (z_1^2,z_1z_2^2,z_2^3) : (z_2) = (z_1^2,z_1z_2,z_2^2); & \\
	{\rm Tr}^1_{z_1}(X) : & \quad (z_1^2,z_1z_2^2,z_2^3) \otimes \bbC\llbracket z_1,z_2 \rrbracket / (z_2) = (z_1^2); \\
	~ \\
	{\rm Res}^2_{z_1}(X) : & \quad (z_1^2,z_1z_2^2,z_2^3) + ((z_1^2,z_1z_2^2,z_2^3) : (z_2^2))\cdot (z_2) = (z_1^2,z_1z_2,z_2^2); & \\
	{\rm Tr}^2_{z_1}(X) : & \quad \left((z_1^2,z_1z_2^2,z_2^3) : (z_2)\right)\otimes \bbC\llbracket z_1,z_2 \rrbracket / (z_2) = (z_1^2); \\
	~ \\
	{\rm Res}^3_{z_1}(X) : & \quad (z_1^2,z_1z_2^2,z_2^3) + ((z_1^2,z_1z_2^2,z_2^3): (z_2^3))\cdot (z_2^2) = (z_1^2,z_2^2); & \\
	{\rm Tr}^3_{z_1}(X) : & \quad \left((z_1^2,z_1z_2^2,z_2^3): (z_2^2)\right)\otimes \bbC\llbracket z_1,z_2 \rrbracket / (z_2) = (z_1).
\end{align*}
\end{minipage}
\begin{minipage}{0.4\textwidth}
\begin{tikzpicture}[line cap=round,line join=round,x=1.0cm,y=1.0cm,scale = 0.6]
\clip(-2.,-1.5) rectangle (4.,10.);
\draw [line width=1.pt,dash pattern=on 1pt off 2pt,domain=-2.:4.] plot(\x,{(--1.0207472463132268-0.*\x)/1.});
\draw [line width=1.pt,dash pattern=on 1pt off 2pt,domain=-2.:4.] plot(\x,{(--5.964822902537174-0.*\x)/1.});
\begin{scriptsize}
\draw [color=black] (0.15798866047282512,3.020747246313228) circle (3.5pt);
\draw [color=black] (0.15798866047282512,2.020747246313227) circle (3.5pt);
\draw [fill=black] (0.15798866047282512,1.020747246313226) circle (4.0pt);
\draw [color=black] (1.1579886604728258,2.020747246313227) circle (3.5pt);
\draw[color=black] (-6.874345903049084,1.514396450654374) node {$f_1$};
\draw [color=black] (0.14984696094040872,7.964822902537176) circle (3.5pt);
\draw [color=black] (0.14984696094040872,6.964822902537174) circle (3.5pt);
\draw [fill=black] (0.14984696094040872,5.964822902537174) circle (4.0pt);
\draw [color=black] (1.1498469609404087,6.964822902537174) circle (3.5pt);
\draw [fill=black] (1.1498469609404087,5.964822902537174) circle (4.0pt);
\draw[color=black] (-6.874345903049084,5.859958317926224) node {$f_2$};
\draw [color=black] (1.1549537540078336,3.01286606005846) circle (3.5pt);
\end{scriptsize}
\node[draw] at (0.5,4.8) {$j = 1,2$};
\node[draw] at (0.4,0.1) {$j = 3$};
\end{tikzpicture}
\end{minipage}
\end{example}
\begin{notation}
	Let $\bbX = X_1 + \ldots+X_s$ be a union of vertically graded schemes with respect to the same base $C$. Then, for any vector $\bfj = (j_1,\ldots,j_s) \in \bbN_{\geq 1}^s$, we denote
	$$
		{\rm Res}^\bfj_C(\bbX) := {\rm Res}^{\bfj_1}_C(\bbX_1) \cap \ldots \cap {\rm Res}^{\bfj_s}_C(\bbX_s), \quad \quad \text{ and } \quad \quad {\rm Tr}^\bfj_C(\bbX) := {\rm Tr}^{\bfj_1}_C(\bbX_1) \cap \ldots \cap {\rm Tr}^{\bfj_s}_C(\bbX_s).
	$$ 
\end{notation}

We are now ready to describe the Horace differential method.  

\begin{proposition}[Horace differential lemma, \mbox{\cite[Proposition 9.1]{AH}}]\label{proposition: Horace}
	Let $\bbX$ be a $0$-dimensional scheme and let $L$ be a line. Let $Y_1,\ldots,Y_s,\widetilde{Y}_1,\ldots,\widetilde{Y_s}$ be $0$-dimensional connected schemes such that $Y_i \simeq \widetilde{Y_i}$, for any $i = 1,\ldots,s$; $\widetilde{Y_i}$ has support on the line $L$ and is vertically graded with base $L$; the support of $\bbY = \bigcup_{i=1}^t Y_i$ and of $\widetilde{\bbY} = \bigcup_{i=1}^t \widetilde{Y_i}$ are generic in the corresponding Hilbert schemes. 
	
	Let $\bfj = (j_1,\ldots,j_s)\in \bbN_{\geq 1}^s$ and $d \in \bbN$.
	
	\begin{enumerate}
	\item If:
	\begin{enumerate}
		\item ${\rm Tr}^1_L(\bbX) + {\rm Tr}^\bfj_L(\widetilde{\bbY})$ imposes independent conditions on $\calO_{\ell}(d)$;
		\item ${\rm Res}^1_L(\bbX) + {\rm Res}^\bfj_L(\widetilde{\bbY})$ imposes independent conditions on $\calO_{\bbP^2}(d-1)$;
	\end{enumerate}
	then, $\bbX + \bbY$ imposes independent conditions on $\calO_{\bbP^2}(d)$.
	\item If:
	\begin{enumerate}
		\item $\calL_{d,\bbP^1}\left({\rm Tr}^1_\ell(\bbX) + {\rm Tr}^\bfj_\ell(\widetilde{\bbY})\right)$ is empty;
		\item $\calL_{d-1,\bbP^2}\left({\rm Res}^1_\ell(\bbX) + {\rm Res}^\bfj_\ell(\widetilde{\bbY})\right)$ is empty;
	\end{enumerate}
	then $\calL_d(\bbX + \bbY)$ is empty.
	\end{enumerate}
\end{proposition}

The latter result contains all our strategy. Given a $0$-dimensional scheme as in Question \ref{question: 0-dim P2} with generic support, we specialize some of the $(3,2)$-points to have support on a line in such a way the arithmetic allows us to use the conditions of Proposition \ref{proposition: Horace}. Such a specialization will be done in one of the different ways explained in Example \ref{example: (3,2)-point}. Recall that, if a specialized scheme has the expected dimension, then, by semicontinuity of the Hilbert function, also the original general scheme has the expected dimension.

In particular, the residues of $(3,2)$-points have very particular structures. For this reason, we introduce the following definitions.

\begin{definition}
	We call {\bf $m$-jet} with {\bf support} at $P$ in the {\bf direction} $L$ the $0$-dimensional scheme defined by the ideal $(\ell,\ell_1^m)$ where $\ell$ is a linear form defining the line $L$ and $(\ell,\ell_1)$ defines the point $P$.
	
	We call {\bf $(m_1,m_2)$-jet} with {\bf support} at $P$ in the {\bf directions} $L_1,L_2$ the $0$-dimensional scheme defined by the ideal $(\ell_1^{m_1},\ell_2^{m_2})$ where $\ell_i$ is a linear form defining the line $L_i$, for $i =1,2$, and $P = L_1 \cap L_2$.
\end{definition}

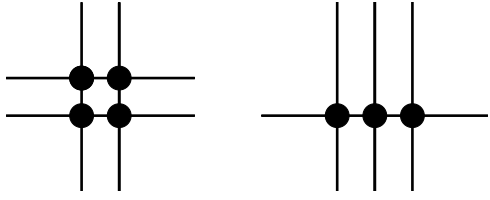
\begin{figure}[h]
\begin{center}
\begin{tikzpicture}[line cap=round,line join=round,x=1.0cm,y=1.0cm]
\clip(5.,5.) rectangle (7.5,7.5);
\draw [line width=1.pt] (1.,5.) -- (1.,7.5);
\draw [line width=1.pt,domain=5.:7.5] plot(\x,{(--6.-0.*\x)/1.});
\draw [line width=1.pt] (1.,5.) -- (1.,7.5);
\draw [line width=1.pt] (6.,5.) -- (6.,7.5);
\draw [line width=1.pt] (6.5,5.) -- (6.5,7.5);
\draw [line width=1.pt] (6.5,5.) -- (6.5,7.5);
\draw [line width=1.pt,domain=5.:7.5] plot(\x,{(--3.25-0.*\x)/0.5});
\begin{scriptsize}
\draw [fill=black] (1.,7.) circle (4.5pt);
\draw [fill=black] (1.,6.) circle (4.5pt);
\draw [fill=black] (6.,6.5) circle (4.5pt);
\draw [fill=black] (2.,6.) circle (4.5pt);
\draw [fill=black] (6.,6.) circle (4.5pt);
\draw [fill=black] (6.5,6.5) circle (4.5pt);
\draw [fill=black] (6.5,6.) circle (4.5pt);
\draw[color=black] (1.0995805946910735,9.785609888363922) node {$f$};
\draw[color=black] (1.0116847258085508,6.162346848873283) node {$i$};
\draw[color=black] (1.0995805946910735,9.785609888363922) node {$j$};
\draw[color=black] (6.099878913341249,9.785609888363922) node {$k$};
\draw[color=black] (6.597955503675544,9.785609888363922) node {$l$};
\draw[color=black] (6.597955503675544,9.785609888363922) node {$h$};
\draw [fill=black] (6.,6.5) circle (4.5pt);
\draw[color=black] (1.0116847258085508,6.435800663174463) node {$g$};
\end{scriptsize}
\end{tikzpicture}
\quad\quad
\begin{tikzpicture}[line cap=round,line join=round,x=1.0cm,y=1.0cm]
\clip(5.,5.) rectangle (8.,7.5);
\draw [line width=1.pt] (1.,5.) -- (1.,8.);
\draw [line width=1.pt] (1.97546,5.) -- (1.97546,8.);
\draw [line width=1.pt,domain=5.:8.] plot(\x,{(--6.-0.*\x)/1.});
\draw [line width=1.pt] (1.,5.) -- (1.,8.);
\draw [line width=1.pt] (6.,5.) -- (6.,8.);
\draw [line width=1.pt] (6.5,5.) -- (6.5,8.);
\draw [line width=1.pt] (7.,5.) -- (7.,8.);
\draw [line width=1.pt] (6.5,5.) -- (6.5,8.);
\begin{scriptsize}
\draw [fill=black] (1.,7.) circle (4.5pt);
\draw [fill=black] (1.,6.) circle (4.5pt);
\draw [fill=black] (2.00409,7.) circle (4.5pt);
\draw [fill=black] (2.,6.) circle (4.5pt);
\draw [fill=black] (6.,6.) circle (4.5pt);
\draw [fill=black] (7.,6.) circle (4.5pt);
\draw [fill=black] (6.5,6.) circle (4.5pt);
\draw[color=black] (1.0995805946910735,9.785609888363922) node {$f$};
\draw[color=black] (2.1152661906668904,9.785609888363922) node {$g$};
\draw[color=black] (1.0116847258085508,6.162346848873283) node {$i$};
\draw[color=black] (1.0995805946910735,9.785609888363922) node {$j$};
\draw[color=black] (6.099878913341249,9.785609888363922) node {$k$};
\draw[color=black] (6.597955503675544,9.785609888363922) node {$l$};
\draw[color=black] (6.871409317976725,9.785609888363922) node {$m$};
\draw[color=black] (6.597955503675544,9.785609888363922) node {$h$};
\end{scriptsize}
\end{tikzpicture}
\end{center}
\caption{The representation of a $(2,2)$-jet and a $3$-jet as vertically graded schemes.}
\end{figure}

Since $(2,2)$-points will be crucial in our computations, we analyse further their structure with the following two lemmas that are represented also in Figure \ref{fig: (3,2)-point structure}.

\begin{lemma}\label{lemma: degeneration}
	For $\lambda > 0$, let $L_1(\lambda)$ and $L_2(\lambda)$ two families of lines, defined by $\ell_1(\lambda)$ and $\ell_2(\lambda)$, respectively, passing through a unique point. Assume also that $m = \lim_{\lambda \rightarrow 0} \ell_i(\lambda)$, for $i = 1,2$, i.e., the families $L_i(\lambda)$ degenerate to the line $M = \{m=0\}$ when $\lambda$ runs to $0$. Fix a generic line $N = \{n = 0\}$ and consider
	$$
		J_1(\lambda) = (\ell_1(\lambda),n^2) \quad \text{ and } \quad J_2(\lambda) = (\ell_2(\lambda),n^2).
	$$
	Then, the limit for $\lambda \rightarrow 0$ of the scheme $\bbX(\lambda) = J_1(\lambda) + J_2(\lambda)$ is the $(2,2)$-jet defined by $(m^2,n^2)$.
\end{lemma}
\begin{proof}
	We may assume that	
	$$
		I(J_1(\lambda)) = (z_1+\lambda z_0,z^2_2) \quad \text{and} \quad 
		I(J_2(\lambda)) = (z_1-\lambda z_0,z^2_2).
	$$
	Hence, $M$ is the line $\{z_1 = 0\}$. Then, the limit for $\lambda \rightarrow 0$ of the scheme $\bbX(\lambda)$ is given by
	$$
		\lim_{\lambda \rightarrow 0} I(\bbX(\lambda)) = \lim_{\lambda \rightarrow 0} \left[(z^2_2,z_1+\lambda z_0)  \cap (z^2_2,z_1-\lambda z_0)\right] = \lim_{\lambda \rightarrow 0} (z_2^2,z_1^2 - \lambda^2z_0^2) = (z_1^2,z_2^2).
	$$
\end{proof}
\begin{lemma}\label{lemma: residue}
	Let $J$ be a $(2,2)$-jet, defined by the ideal $(\ell_1^2,\ell_2^2)$, with support at the point $P$. Let $L_i = \{\ell_i=0\}$, for $i = 1,2$. Then:
	\begin{enumerate}
		\item the residue of $J$ with respect to $L_1$ ($L_2$, respectively) is a $2$-jet with support at $P$ and direction $L_1$ ($L_2$, respectively);
		\item the residue of $J$ with respect to a line $L = \{\alpha \ell_1 + \beta \ell_2 = 0\}$ passing through $P$ different from $L_1$ and $L_2$ is a $2$-jet with support at $P$ and direction the line $\{\alpha \ell_1 - \beta \ell_2 = 0\}$.
	\end{enumerate}
\end{lemma}
\begin{proof}
	\begin{enumerate}
	\item If we consider the residue with respect to $\{\ell_1=0\}$, we get
	$$
		(\ell_1^2,\ell_2^2) : (\ell_1) = (\ell_1,\ell_2^2).
	$$
	Analogously, for the line $\{\ell_2 = 0\}$.
	\item If we consider the residue with respect to the line $\{\alpha \ell_1 + \beta\ell_2 = 0\}$, we get
	$$
		(\ell_1^2,\ell_2^2) : (\alpha \ell_1 + \beta \ell_2) = (\ell_1^2, \alpha\ell_1 - \beta \ell_2).
	$$
	\end{enumerate}
\end{proof}
\begin{figure}[h]
\begin{subfigure}[t]{0.5\textwidth}
        \centering
        \includegraphics[height=1.6in]{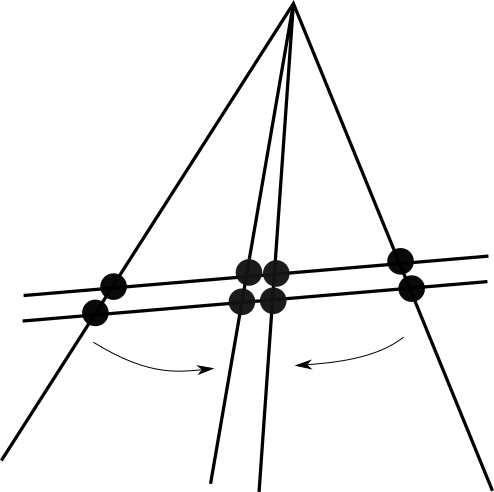}
        \caption{The degeneration of Lemma \ref{lemma: degeneration}.}
    \end{subfigure}%
    ~ 
    \begin{subfigure}[t]{0.5\textwidth}
        \centering
        \includegraphics[height=1.3in]{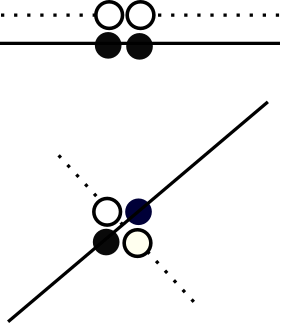}
        \caption{The two different residues of Lemma \ref{lemma: residue}.}
    \end{subfigure}
    \caption{The structure of a $(2,2)$-jet.}
    \label{fig: (3,2)-point structure}
\end{figure}

The construction of $(2,2)$-jets as degeneration of $2$-jets, as far as we know, is a type of degeneration that has not been used before in the literature. Similarly as regard the fact that the structure of the residues of $(2,2)$-jets depend on the direction of the lines. These two facts will be crucial for our computation and we believe that these constructions might be used to attack also other similar problems on linear systems. 

\section{Lemmata}\label{sec: lemmata}

\subsection{Subabundance and superabundance}\label{sec: super- and sub-abundance}
The following result is well-known for the experts in the area and can be found
in several papers in the literature. We explicitly recall it for convenience of the reader.
\begin{lemma}\label{lemma: super- and sub-abundance}
Let $\bbX' \subset \bbX \subset \bbX''\subset\bbP^2$ be $0$-dimensional schemes. Then:
\begin{enumerate}
	\item if $\bbX$ imposes independent conditions on $\calO_{\bbP^2}(d)$, then also $\bbX'$ does;
	\item if $\calL_{d}(\bbX)$ is empty, then also $\calL_{d}(\bbX'')$ is empty.
\end{enumerate}
\end{lemma}
In Question \ref{question: 0-dim P2}, we consider, for any positive integers $a,b$ and $s$, the scheme
$$
	\bbX_{a,b;s} = aQ_1 + bQ_2 + Y_1 + \ldots Y_s \subset \bbP^2,
$$
where the $Y_i$'s are general $(3,2)$-points with support at general points $\{P_1,\ldots,P_s\}$ and general directions. The previous lemma suggests that, fixed $a,b$, there are two critical values to be considered firstly, i.e.,
$$
	s_1 = \left\lfloor \frac{(a+1)(b+1)}{5} \right\rfloor \quad \quad \text{ and } \quad \quad 	s_2 = \left\lceil \frac{(a+1)(b+1)}{5} \right\rceil;
$$
namely, $s_1$ is the largest number of $(3,2)$-points for which we expect to have \textit{subabundance}, i.e., where we expect to have positive virtual dimension, and $s_2$ is the smallest number of $(3,2)$-points where we expect to have \textit{superabundance}, i.e., where we expect that the virtual dimension is negative. If we prove that the dimension of linear system $\calL_{a+b}(\bbX_{a,b;s})$ is as expected for $s = s_1$ and $s = s_2$, then, by Lemma \ref{lemma: super- and sub-abundance}, we have that it holds for any $s$.

\subsection{Low bi-degrees}
Now, we answer to Question \ref{question: 0-dim P2} for $b = 1,2$. These will be the base cases of our inductive approach to solve the problem in general. Recall that $a,b$ are positive integers such that $ab > 1$; see Remark \ref{rmk: (a,b) = (1,1)}.
\begin{lemma}\label{lemma: b = 1}
	Let $a > b = 1$ be a positive integer. Then,
	$$
		\dim\mathcal{L}_{a+1}(\bbX_{a,1;s}) = \max \{0,~ 2(a+1) - 5s\}.
	$$
\end{lemma}
\begin{proof}
	First of all, note that $a \geq s_2$. Indeed, 
	$$
		a \geq \frac{2(a+1)}{5} \quad \Longleftrightarrow \quad 3a \geq 2.
	$$
	Now, if $s \leq s_2$, we note that every line $\overline{Q_1P_i}$ is contained in the base locus of $\mathcal{L}_{a+1}(\bbX_{a,1;s})$. Hence,
	$$
		\dim\mathcal{L}_{a+1}(\bbX_{a,1;s}) = \dim\mathcal{L}_{a+1-s}(\bbX'),
	$$
	where $\bbX' = (a-s)Q_1 + Q_2 + 2P_1 + \ldots + 2P_s$. By \cite{CGG05-SegreVeronese},
	$$
		\dim\mathcal{L}_{a+1}(\bbX_{a,1;s}) = \dim\mathcal{L}_{a+1-s}(\bbX') = \max\{0,~2(a+1-s) - 3s\}.
	$$
\end{proof}

Now, we prove the case $b = 2$ which is the crucial base step for our inductive procedure. In order to make our construction to work smoothly, we need to consider separately the following easy case.

\begin{lemma}\label{lemma: a = 2}
	Let $a = b = 2$. Then, 
	 $$\dim\mathcal{L}_{4}(\bbX_{2,2;s}) = \max\{0, 9 - 5s\}.$$
\end{lemma}
\begin{proof}
	For $s = s_1 = 1$, then it follows easily because the scheme $2Q_1 + 2Q_2 + 3P_1 \supset \bbX_{2,2;1}$ imposes independent conditions on quartics. For $s = s_2 = 2$, we specialize the directions of the $(3,2)$-points supported at the $P_i$'s to be along the lines $\overline{Q_1P_i}$, respectively. Now, the lines $\overline{Q_1P_i}$ are fixed components and we can remove them. We remain with the linear system $\calL_2(2Q_2 + J_1 + J_2)$, where the $J_i$'s are $2$-jets contained in the lines $\overline{Q_1P_i}$, respectively. Since both lines $\overline{Q_2P_1}$ and $\overline{Q_2P_2}$ are fixed components for this linear system, we conclude that the linear system has to be empty. 
\end{proof} 

\begin{lemma}\label{lemma: b = 2}
	Let $a > b = 2$ be a positive integer. Then,
	$$
		\dim\mathcal{L}_{a+2}(\bbX_{a,2;s}) = \max \{0,~ 3(a+1) - 5s\}.
	$$
\end{lemma}
\begin{proof}
	We split the proof in different steps. Moreover, in order to help the reader in following the constructions, we include figures showing the procedure in the case of $a = 15$. 

	\medskip
	\noindent {\sc Step 1.} Note that since $a > 2$, then $a \geq \frac{3(a+1)}{5}$ which implies $a \geq \left\lceil \frac{3(a+1)}{5} \right\rceil = s_2$.
	
	We specialize the $(3,2)$-points with support at the $P_i$'s to have direction along the lines $\overline{Q_1P_i}$, respectively. In this way, for any $s \leq s_2$, every line $\overline{Q_1P_i}$ is a fixed component of the linear system $\mathcal{L}_{a+2}(\bbX_{a,2;s})$ and can be removed, i.e.,
	$$
		\dim\mathcal{L}_{a+2}(\bbX_{a,2;s}) = \dim\mathcal{L}_{(a-s)+2}(\bbX'),
	$$
	where $\bbX' = a'Q_1 + 2Q_2 + J_1+\ldots+J_s$, where $a' = a-s$ and $J_i$ is a $2$-jet contained in $\overline{Q_1P_i}$, for $i = 1,\ldots,s$. Now, as suggested by Lemma \ref{lemma: super- and sub-abundance}, we consider two cases: $s = s_1$ and $s = s_2$. See Figure \ref{fig: Lemma_1}.
	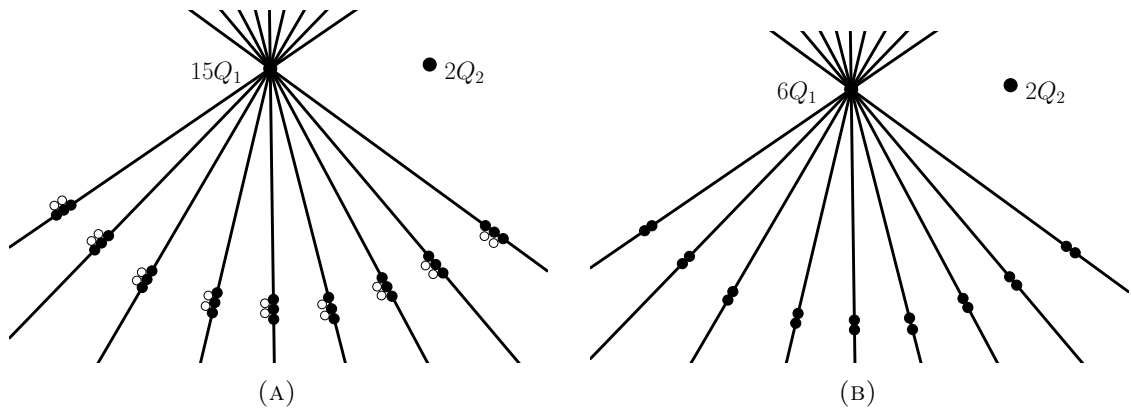
\begin{figure}[h!]
	\centering
	\begin{subfigure}[b]{0.46\textwidth}
	\centering
	\scalebox{0.55}{
	\begin{tikzpicture}[line cap=round,line join=round,x=1.0cm,y=1.0cm]
\clip(-5.,-4.) rectangle (8.,4.5);
\draw [line width=2.pt,domain=-5.:8.] plot(\x,{(-10.638532277887979-3.303223402088166*\x)/-4.817007322775031});
\draw [line width=2.pt,domain=-5.:8.] plot(\x,{(--7.162--4.22*\x)/4.08});
\draw [line width=2.pt,domain=-5.:8.] plot(\x,{(-2.608-5.1*\x)/-2.98});
\draw [line width=2.pt,domain=-5.:8.] plot(\x,{(--3.204-5.66*\x)/-1.34});
\draw [line width=2.pt,domain=-5.:8.] plot(\x,{(--7.814-5.82*\x)/0.08});
\draw [line width=2.pt,domain=-5.:8.] plot(\x,{(--12.1506341821624-5.812947455129581*\x)/1.4818717711270788});
\draw [line width=2.pt,domain=-5.:8.] plot(\x,{(--15.606-5.28*\x)/2.82});
\draw [line width=2.pt,domain=-5.:8.] plot(\x,{(--18.562-4.74*\x)/4.});
\draw [line width=2.pt,domain=-5.:8.] plot(\x,{(--21.9276929872492-3.9572957108762044*\x)/5.41393824616456});
\begin{scriptsize}
\draw [fill=black] (1.3,3.1) circle (4.5pt);
\draw[color=black] (0.,3.0) node {\LARGE $15Q_1$};
\draw [fill=black] (5.153549103316712,3.1969296513445316) circle (4.5pt);
\draw[color=black] (6,3.0) node {\LARGE $2Q_2$};
\draw [fill=black] (-1.68,-2.) circle (3.5pt);
\draw [fill=black] (1.38,-2.72) circle (3.5pt);
\draw [fill=black] (5.3,-1.64) circle (3.5pt);
\draw [fill=black] (-3.5170073227750307,-0.2032234020881659) circle (3.5pt);
\draw [fill=black] (6.71393824616456,-0.8572957108762045) circle (3.5pt);
\draw [fill=black] (-2.78,-1.12) circle (3.5pt);
\draw [fill=black] (-0.04,-2.56) circle (3.5pt);
\draw [fill=black] (2.781871771127079,-2.7129474551295814) circle (3.5pt);
\draw [fill=black] (4.12,-2.18) circle (3.5pt);
\draw [fill=black] (-3.692809695702354,-0.3237784424864891) circle (3.5pt);
\draw [fill=black] (-3.868434085104564,-0.4442114321371302) circle (3.5pt);
\draw [color=black] (-3.730728634124213,-0.0883446707358218) circle (3.0pt);
\draw [color=black] (-3.9189726904653908,-0.2154369516225996) circle (3.0pt);
\draw [fill=black] (-2.942781799741943,-1.2883674497330886) circle (3.5pt);
\draw [fill=black] (-2.6092513870360494,-0.9433923660029729) circle (3.5pt);
\draw [color=black] (-3.0209331459720903,-1.0713532114933095) circle (3.0pt);
\draw [color=black] (-2.8529828829557946,-0.9017441734600915) circle (3.0pt);
\draw [fill=black] (-1.5606462502768503,-1.7957368712791735) circle (3.5pt);
\draw [fill=black] (-1.7943395169594518,-2.195681723655438) circle (3.5pt);
\draw [color=black] (-1.8065098757689753,-1.8274702952022752) circle (3.0pt);
\draw [color=black] (-1.9272567612136082,-2.028715104276663) circle (3.0pt);
\draw [fill=black] (0.015860309514192927,-2.324052722499752) circle (3.5pt);
\draw [fill=black] (-0.09866682693045781,-2.8078016719599934) circle (3.5pt);
\draw [color=black] (-0.26363300619866475,-2.6458658521047855) circle (3.0pt);
\draw [color=black] (-0.19655140317386866,-2.390955760610561) circle (3.0pt);
\draw [fill=black] (1.3833771614291572,-2.965688493971196) circle (3.5pt);
\draw [fill=black] (1.3767723595596177,-2.4851891579622047) circle (3.5pt);
\draw [color=black] (1.171913298531972,-2.5787842490799897) circle (3.0pt);
\draw [color=black] (1.171913298531972,-2.8202780199692548) circle (3.0pt);
\draw [fill=black] (2.711156255843747,-2.435551271051) circle (3.5pt);
\draw [fill=black] (2.843016170049994,-2.9527989625539086) circle (3.5pt);
\draw [color=black] (2.553794320842772,-2.605616890289908) circle (3.0pt);
\draw [color=black] (2.620875923867568,-2.8739433023890917) circle (3.0pt);
\draw [fill=black] (4.0029287318171365,-1.9608027319129353) circle (3.5pt);
\draw [fill=black] (4.241892133233513,-2.4082235686074287) circle (3.5pt);
\draw [color=black] (3.882010060733735,-2.176294630931214) circle (3.0pt);
\draw [color=black] (3.9893406255734085,-2.390955760610561) circle (3.0pt);
\draw [fill=black] (5.129811402518572,-1.4383265119845086) circle (3.5pt);
\draw [fill=black] (5.473324931017845,-1.8453900432561468) circle (3.5pt);
\draw [color=black] (5.053640490627722,-1.6646197621124534) circle (3.0pt);
\draw [color=black] (5.237141378826413,-1.8743350629109574) circle (3.0pt);
\draw [fill=black] (6.499833566905019,-0.700796783406437) circle (3.5pt);
\draw [color=black] (6.477093465553725,-0.9548200323041947) circle (3.0pt);
\draw [color=black] (6.700006200246276,-1.1220045833236065) circle (3.0pt);
\draw [fill=black] (6.929793986843945,-1.0150745694296224) circle (3.5pt);
\end{scriptsize}
\end{tikzpicture}
}
\caption{}
\end{subfigure}
\begin{subfigure}[b]{0.46\textwidth}
\centering
	\scalebox{0.55}{
\begin{tikzpicture}[line cap=round,line join=round,x=1.0cm,y=1.0cm]
\clip(-5.,-3.5) rectangle (8.,4.5);
\draw [line width=2.pt,domain=-5.:8.] plot(\x,{(-10.638532277887979-3.303223402088166*\x)/-4.817007322775031});
\draw [line width=2.pt,domain=-5.:8.] plot(\x,{(--7.162--4.22*\x)/4.08});
\draw [line width=2.pt,domain=-5.:8.] plot(\x,{(-2.608-5.1*\x)/-2.98});
\draw [line width=2.pt,domain=-5.:8.] plot(\x,{(--3.204-5.66*\x)/-1.34});
\draw [line width=2.pt,domain=-5.:8.] plot(\x,{(--7.814-5.82*\x)/0.08});
\draw [line width=2.pt,domain=-5.:8.] plot(\x,{(--12.1506341821624-5.812947455129581*\x)/1.4818717711270788});
\draw [line width=2.pt,domain=-5.:8.] plot(\x,{(--15.606-5.28*\x)/2.82});
\draw [line width=2.pt,domain=-5.:8.] plot(\x,{(--18.562-4.74*\x)/4.});
\draw [line width=2.pt,domain=-5.:8.] plot(\x,{(--21.9276929872492-3.9572957108762044*\x)/5.41393824616456});
\begin{scriptsize}
\draw [fill=black] (1.3,3.1) circle (4.5pt);
\draw[color=black] (0.,3.) node {\LARGE $6Q_1$};
\draw [fill=black] (5.153549103316712,3.1969296513445316) circle (4.5pt);
\draw[color=black] (6,3) node {\LARGE $2Q_2$};
\draw [fill=black] (-1.68,-2.) circle (3.5pt);
\draw [fill=black] (1.38,-2.72) circle (3.5pt);
\draw [fill=black] (5.3,-1.64) circle (3.5pt);
\draw [fill=black] (-3.5170073227750307,-0.2032234020881659) circle (3.5pt);
\draw [fill=black] (6.71393824616456,-0.8572957108762045) circle (3.5pt);
\draw [fill=black] (-2.78,-1.12) circle (3.5pt);
\draw [fill=black] (-0.04,-2.56) circle (3.5pt);
\draw [fill=black] (2.781871771127079,-2.7129474551295814) circle (3.5pt);
\draw [fill=black] (4.12,-2.18) circle (3.5pt);
\draw [fill=black] (-3.692809695702354,-0.3237784424864891) circle (3.5pt);
\draw [fill=black] (-2.6092513870360494,-0.9433923660029729) circle (3.5pt);
\draw [fill=black] (-1.5606462502768503,-1.7957368712791735) circle (3.5pt);
\draw [fill=black] (0.015860309514192927,-2.324052722499752) circle (3.5pt);
\draw [fill=black] (1.3767723595596177,-2.4851891579622047) circle (3.5pt);
\draw [fill=black] (2.711156255843747,-2.435551271051) circle (3.5pt);
\draw [fill=black] (4.0029287318171365,-1.9608027319129353) circle (3.5pt);
\draw [fill=black] (5.129811402518572,-1.4383265119845086) circle (3.5pt);
\draw [fill=black] (6.499833566905019,-0.700796783406437) circle (3.5pt);
\end{scriptsize}
\end{tikzpicture}
}
\caption{}
\end{subfigure}
\caption{As example, we consider the case $a = 15$ and $s = s_1 = \left\lfloor \frac{16\cdot 3}{5} \right\rfloor = 9$. Here, we represent: ({\sc A}) the first specialization and ({\sc B}) the reduction explained in Step 1.}
\label{fig: Lemma_1}
	\end{figure}
	
	\medskip
	We fix $a = 5k + c$, with $ 0 \leq c \leq 4$.
	
	\medskip
	\noindent {\sc Step 2: case $s = s_1$.} Consider
	$$
			s_1 = \left\lfloor \frac{3(5k+c+1)}{5} \right\rfloor =
			\begin{cases}
				3k & \text{ for } c = 0; \\
				3k+1 & \text{ for } c = 1,2; \\
				3k+2 & \text{ for } c = 3; \\
				3k+3 & \text{ for } c = 4. 
			\end{cases}
	$$
		Note that the expected dimension is 
	$$
		{\it exp}.\dim\mathcal{L}_{a+2}(\bbX_{a,2;s_1}) = 
		\begin{cases}
			3 & \text{ for } c = 0; \\
			1 & \text{ for } c = 1; \\
			4 & \text{ for } c = 2; \\
			2 & \text{ for } c = 3; \\
			0 & \text{ for } c = 4. \\
		\end{cases}
	$$
	Let $A$ be a reduced set of points of cardinality equal to the expected dimension. Then, it is enough to show that 
	$$
		\dim\mathcal{L}_{a'+2}(\bbX'+A) = 0.
	$$
	where
	$$
		a' = a-s_1 = 
		\begin{cases}
				2k & \text{ for } c = 0,1; \\
				2k+1 & \text{ for } c = 2,3,4.
		\end{cases}
	$$
	Let
	$$
		t_1 = 
		\begin{cases}
			k & \text{ for } c = 0,2;\\
			k+1 & \text{ for } c = 1,3,4.
		\end{cases}
	$$
	Note that $s_1 \geq 2t_1$. Now, by using Lemma \ref{lemma: degeneration}, we specialize $2t_1$ lines in such a way that, for $i = 1,\ldots,t_1$, we have:
	\begin{itemize}
		\item the lines $\overline{Q_1P_{2i-1}}$ and $\overline{Q_1P_{2i}}$ both degenerate to a general line $R_i$ passing through $Q_1$;
		\item the point $P_{2i-1}$ and the point $P_{2i}$ both degenerate to a general point $\widetilde{P}_i$ on $R_i$.
	\end{itemize}
	In this way, from the degeneration of $J_{2i-1}$ and $J_{2i}$, we obtain the $(2,2)$-jet $W_i$ defined by the scheme-theoretic intersection $2R_i \cap 2\overline{P_{2i-1}P_{2i}}$, for any $i = 1,\ldots,t_1$. Note that the directions $\overline{P_{2i-1}P_{2i}}$ are generic, for any $i$; see Figure \ref{fig: Lemma_2}. Note that, these $(2,2)$-jets have general directions.
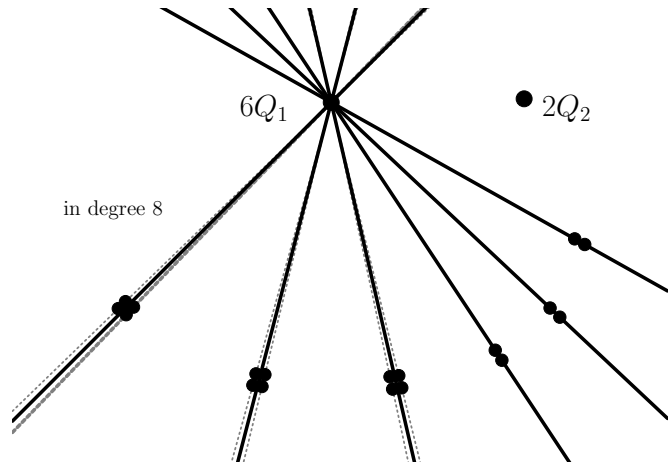
\begin{figure}[h]
\centering
\scalebox{0.67}{
\begin{tikzpicture}[line cap=round,line join=round,x=1.0cm,y=1.0cm]
\clip(-5.,-4.) rectangle (8.,5.);
\draw [line width=1.pt,dotted,domain=-5.:8.,color=gray] plot(\x,{(-7.514031000906525-3.9569717098083377*\x)/-4.082209051951633});
\draw [line width=2.pt,dotted,domain=-5.:8.,color=gray] plot(\x,{(--7.129839324208257--4.216172564101999*\x)/4.06924900923695});
\draw [line width=1.pt,dotted,domain=-5.:8.,color=gray] plot(\x,{(--2.5853698366721902-5.615857177287773*\x)/-1.555000722588424});
\draw [line width=1.pt,dotted,domain=-5.:8.,color=gray] plot(\x,{(--3.1869066081769386-5.641777262717139*\x)/-1.3735601245828604});
\draw [line width=1.pt,dotted,domain=-5.:8.,color=gray] plot(\x,{(--11.358054097920721-5.693617433575872*\x)/1.2184484183537643});
\draw [line width=1.pt,dotted,domain=-5.:8.,color=gray] plot(\x,{(--11.890854945651789-5.667697348146506*\x)/1.3998890163593283});
\draw [line width=2.pt,domain=-5.:8.] plot(\x,{(--17.373726988940465-5.120315184660661*\x)/3.3860744845036743});
\draw [line width=2.pt,domain=-5.:8.] plot(\x,{(--19.832137077686145-4.265077706727545*\x)/4.5352998454763025});
\draw [line width=2.pt,domain=-5.:8.] plot(\x,{(--19.464113343222074-2.8218644627154124*\x)/5.029734012406387});
\draw [line width=2.pt,domain=-5.:8.] plot(\x,{(-12.14777503007683-6.7563409361798845*\x)/-6.751977851176356});
\draw [line width=2.pt,domain=-5.:8.] plot(\x,{(--3.388023713470546-6.587860380889005*\x)/-1.7105212351646215});
\draw [line width=2.pt,domain=-5.:8.] plot(\x,{(--14.585189910082079-7.119222132191011*\x)/1.6461298279383072});
\begin{scriptsize}
\draw [fill=black] (1.325920085429366,3.125920085429366) circle (4.5pt);
\draw[color=black] (0.,3.) node {\LARGE $6Q_1$};
\draw [fill=black] (5.153549103316712,3.1969296513445316) circle (4.5pt);
\draw[color=black] (6,3.) node {\LARGE $2Q_2$};
\draw[color = black] (-3,1) node {\large in degree $8$};
\draw [fill=black] (-0.22908063715905802,-2.489937091858407) circle (3.5pt);
\draw [fill=black] (2.5443685037831303,-2.5676973481465053) circle (3.5pt);
\draw [fill=black] (5.861219930905668,-1.1391576212981789) circle (3.5pt);
\draw [fill=black] (-2.756288966522267,-0.8310516243789717) circle (3.5pt);
\draw [fill=black] (6.355654097835753,0.30405562271395387) circle (3.5pt);
\draw [fill=black] (-2.7433289238075838,-1.0902524786726335) circle (3.5pt);
\draw [fill=black] (-0.04764003915349429,-2.515857177287773) circle (3.5pt);
\draw [fill=black] (2.7258091017886943,-2.5417772627171393) circle (3.5pt);
\draw [fill=black] (4.71199456993304,-1.9943950992312944) circle (3.5pt);
\draw [fill=black] (-2.898985250235947,-0.9693701549990218) circle (3.5pt);
\draw [fill=black] (-2.5973158622124415,-0.9389674963821508) circle (3.5pt);
\draw [fill=black] (-0.16680037975532125,-2.265013060979562) circle (3.5pt);
\draw [fill=black] (0.009619284499902792,-2.280669547358283) circle (3.5pt);
\draw [fill=black] (2.491906026984575,-2.322548462063668) circle (3.5pt);
\draw [fill=black] (2.666250534361207,-2.300643908055881) circle (3.5pt);
\draw [fill=black] (4.580296319385906,-1.7952451528388482) circle (3.5pt);
\draw [fill=black] (5.668255850940832,-0.9576907230338629) circle (3.5pt);
\draw [fill=black] (6.156743519524302,0.4156517211976407) circle (3.5pt);
\end{scriptsize}
\end{tikzpicture}
}
\caption{Since $a = 15$, we have $a = 5\cdot 3 + 0$ and $t_1 = 3$. Hence, we specialize three pairs of lines to be coincident and we construct three $(2,2)$-jets.}
\label{fig: Lemma_2}
\end{figure}	
	
	Moreover, we specialize further as follows; see Figure \ref{fig: Lemma_3}:
	\begin{itemize}
		\item every scheme $W_i$, with support on $R_i$, for $i = 1,\ldots,t_1$, in such a way that all $\widetilde{P}_i$'s are collinear and lie on a line $L$;
		\item if $c = 0,1,2,3$, we specialize $A$ (the reduced part) to lie on $L$;
		\item if $c = 4$, we specialize the double point $2Q_2$ to have support on $L$.
	\end{itemize}
		Note that this can be done because
	$$
		s_1 - t_1 = 
		\begin{cases}
			k & \text{for } c =0,3; \\
			k-1 & \text{for } c = 1; \\
			k+1 & \text{for } c = 2,4;
		\end{cases}
	$$
	and, for $c = 1$, we may assume $k \geq 1$, since $a \geq 2$. 
	\begin{figure}[h!]
	\centering
	\scalebox{0.75}{
	\begin{tikzpicture}[line cap=round,line join=round,x=1.0cm,y=1.0cm]
\clip(-4.,-3.5) rectangle (8.5,4.5);
\draw [line width=2.pt,domain=-4.:8.5] plot(\x,{(--6.606070669096441-2.653365106436809*\x)/0.88280517651701});
\draw [line width=2.pt,domain=-4.:8.5] plot(\x,{(--9.676427200671663-2.92575777177952*\x)/1.7123646573334501});
\draw [line width=2.pt,domain=-4.:8.5] plot(\x,{(--8.795388688536372-2.0466723518098617*\x)/1.8237980204281958});
\draw [line width=2.pt,dotted,domain=-4.:8.5, color = gray] plot(\x,{(--3.6452140416492202-6.69929374398375*\x)/-1.7971916286827576});
\draw [line width=2.pt,domain=-4.:8.5] plot(\x,{(-0.3695083120389595--0.028408168922945176*\x)/0.3571005711237858});
\draw [->,line width=2.pt] (0.3289024955258037,-1.2460623060398182) -- (-0.007560515160232151,-0.4065304409771011);
\draw [line width=2.pt,dotted,domain=-4.:8.5, color = gray] plot(\x,{(--8.607111108960233-5.488725122958664*\x)/0.2637309371017562});
\draw [->,line width=2.pt] (1.4905991442246822,-1.1370380434870988) -- (1.8372807182972242,-0.33224153224727093);
\draw [->,line width=2.pt] (-1.6047720528515872,-1.3227603153116745) -- (-1.8647832334059937,-0.5798712280133718);
\draw [line width=2.pt,dotted,domain=-4.:8.5, color = gray] plot(\x,{(-4.489159349354937-5.241095427192564*\x)/-3.673581225579258});
\begin{scriptsize}
\draw [fill=black] (1.4125904789475021,3.2373534485241113) circle (4.5pt);
\draw[color=black] (0.5,3) node {\LARGE $6Q_1$};
\draw [fill=black] (5.153549103316712,3.1969296513445316) circle (4.5pt);
\draw[color=black] (6,3) node {\LARGE $2Q_2$};
\draw[color=black] (7,0.5) node {\LARGE $A$};
\draw[color = black] (-3,1) node {\large in degree $8$};
\draw [fill=black] (0.10387284793451354,-1.0256046803923533) circle (3.5pt);
\draw [fill=black] (1.4905991442246822,-1.1370380434870988) circle (3.5pt);
\draw [fill=black] (3.1249551362809522,0.3115956767445914) circle (3.5pt);
\draw [fill=black] (-1.716205415946333,-1.0008417108157432) circle (3.5pt);
\draw [fill=black] (3.236388499375698,1.1906810967142496) circle (3.5pt);
\draw [fill=black] (-1.6047720528515872,-1.3227603153116745) circle (3.5pt);
\draw [fill=black] (0.3289024955258037,-1.2460623060398182) circle (3.5pt);
\draw [fill=black] (1.7134658704141736,-0.901789832509303) circle (3.5pt);
\draw [fill=black] (2.295395655464512,0.5839883420873023) circle (3.5pt);
\draw [fill=black] (-1.8498750949904104,-1.181907592294197) circle (3.5pt);
\draw [fill=black] (-1.4927745238666246,-1.1534994233712519) circle (3.5pt);
\draw [fill=black] (0.16974203912965335,-0.8110453305399878) circle (3.5pt);
\draw [fill=black] (0.38458241343333954,-1.0157042832708876) circle (3.5pt);
\draw [fill=black] (1.4867785796410165,-0.9227971620403612) circle (3.5pt);
\draw [fill=black] (1.696281834687383,-0.6653890232810649) circle (3.5pt);
\draw [fill=black] (2.227862500801172,0.7869664406279986) circle (3.5pt);
\draw [fill=black] (3.0173602157974577,0.4954330816223558) circle (3.5pt);
\draw [fill=black] (3.091602635533084,1.3531602863725496) circle (3.5pt);
\draw [fill=black] (6.160542692533723,-0.5446604970705151) circle (3.5pt);
\draw [fill=black] (6.913021105428445,-0.48479911461147396) circle (3.5pt);
\draw [fill=black] (7.712198191933535,-0.42122274563133616) circle (3.5pt);
\end{scriptsize}
\end{tikzpicture}
	}
	\caption{Since $a = 15 = 5\cdot 3 + 0$, we have that $|A| = 3$. Then, we specialize the scheme in such a way the $(2,2)$-jets have support collinear with the support of the set $A$. Recall that they have general directions.}
	\label{fig: Lemma_3}
	\end{figure}
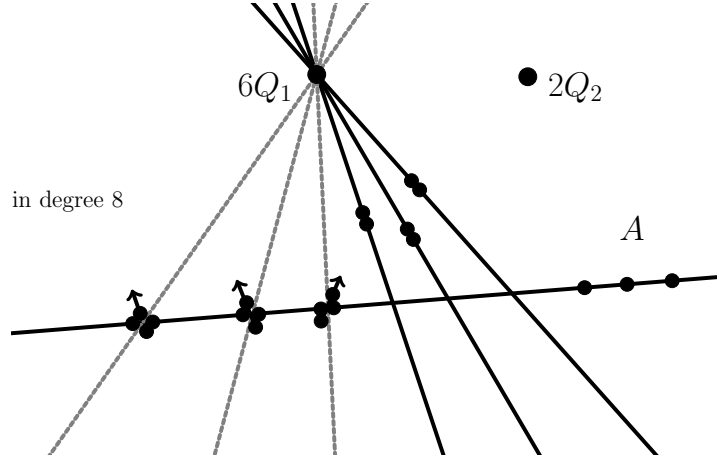
	
	By abuse of notation, we call again ${\bbX}$ the scheme obtained after such a specialization.
	In this way, we have that 
	$$
		\deg({\bbX}\cap L) \quad = \quad
		\begin{cases}
			2k + 3 & \text{for } c = 0,1; \\
			2k + 4 & \text{for } c = 2,3,4.
		\end{cases}  \quad = \quad a'+3.
	$$
	Therefore, the line $L$ becomes a fixed component of the linear system and can be removed, see Figure \ref{fig: Lemma_4}, i.e.,
	$$
		\dim\mathcal{L}_{a'+2}({\bbX}) = \dim\mathcal{L}_{a'+1}(\bbJ),
	$$
	with $$
	\bbJ = 
	\begin{cases}
		a'Q_1 + 2{Q}_2 +  {J}_1 + \ldots + {J}_{t_1} + J_{2t_1+1} + \ldots + J_{s_1}, & \text{ for } c = 0,1,2,3; \\
		a'Q_1 + {Q}_2 +  {J}_1 + \ldots + {J}_{t_1} + J_{2t_1+1} + \ldots + J_{s_1}, & \text{ for } c = 4; \\
	\end{cases}
	$$
	where
	\begin{itemize}
		\item
		by Lemma \ref{lemma: residue}, ${J}_i$ is a $2$-jet with support on $\widetilde{P}_i \in L$ with general direction, for all $i = 1,\ldots,t_1$;
		\item
		$J_i$ is a $2$-jet contained in the line $\overline{Q_1P_i}$, for $i = 2t_1+1,\ldots,s_1$.
	\end{itemize}
	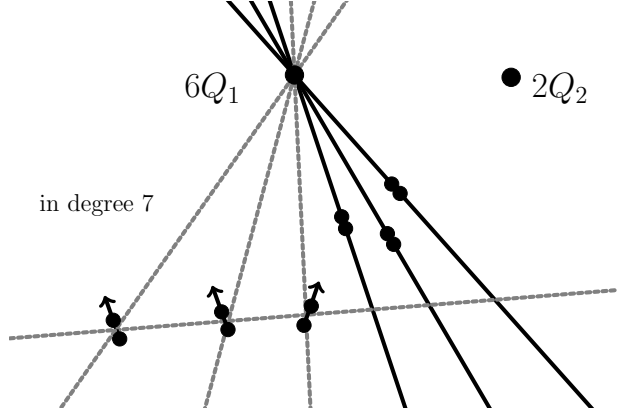
\begin{figure}[h]
	\centering
	\scalebox{0.77}{
	\begin{tikzpicture}[line cap=round,line join=round,x=1.0cm,y=1.0cm]
\clip(-3.5,-2.5) rectangle (7.,4.5);
\draw [line width=2.pt,domain=-3.5:7.] plot(\x,{(--6.606070669096441-2.653365106436809*\x)/0.88280517651701});
\draw [line width=2.pt,domain=-3.5:7.] plot(\x,{(--9.676427200671663-2.92575777177952*\x)/1.7123646573334501});
\draw [line width=2.pt,domain=-3.5:7.] plot(\x,{(--8.795388688536372-2.0466723518098617*\x)/1.8237980204281958});
\draw [line width=2.pt,dotted,domain=-3.5:7., color = gray] plot(\x,{(--3.6452140416492202-6.69929374398375*\x)/-1.7971916286827576});
\draw [line width=2.pt,dotted,domain=-3.5:7., color = gray] plot(\x,{(-0.3695083120389595--0.028408168922945176*\x)/0.3571005711237858});
\draw [->,line width=2.pt] (0.2648321501824796,-1.1618010130637089) -- (-0.007560515160232151,-0.4065304409771011);
\draw [line width=2.pt,dotted,domain=-3.5:7., color = gray] plot(\x,{(--8.607111108960233-5.488725122958664*\x)/0.2637309371017562});
\draw [->,line width=2.pt] (1.5772695377428179,-1.0875121043338787) -- (1.8372807182972242,-0.33224153224727093);
\draw [->,line width=2.pt] (-1.6047720528515872,-1.3227603153116745) -- (-1.8647832334059937,-0.5798712280133718);
\draw [line width=2.pt,dotted,domain=-3.5:7., color = gray] plot(\x,{(-4.489159349354937-5.241095427192564*\x)/-3.673581225579258});
\begin{scriptsize}
\draw [fill=black] (1.4125904789475021,3.2373534485241113) circle (4.5pt);
\draw[color=black] (0,3) node {\LARGE $6Q_1$};
\draw [fill=black] (5.153549103316712,3.1969296513445316) circle (4.5pt);
\draw[color=black] (6,3) node {\LARGE $2Q_2$};
\draw[color = black] (-2,1) node {\large in degree $7$};
\draw [fill=black] (1.5772695377428179,-1.0875121043338787) circle (3.5pt);
\draw [fill=black] (3.1249551362809522,0.3115956767445914) circle (3.5pt);
\draw [fill=black] (-1.716205415946333,-1.0008417108157432) circle (3.5pt);
\draw [fill=black] (3.236388499375698,1.1906810967142496) circle (3.5pt);
\draw [fill=black] (-1.6047720528515872,-1.3227603153116745) circle (3.5pt);
\draw [fill=black] (0.2648321501824796,-1.1618010130637089) circle (3.5pt);
\draw [fill=black] (2.295395655464512,0.5839883420873023) circle (3.5pt);
\draw [fill=black] (0.15570386993870922,-0.8567726333969978) circle (3.5pt);
\draw [fill=black] (1.7029018983959778,-0.7564612511753777) circle (3.5pt);
\draw [fill=black] (2.227862500801172,0.7869664406279986) circle (3.5pt);
\draw [fill=black] (3.0173602157974577,0.4954330816223558) circle (3.5pt);
\draw [fill=black] (3.0916026355330835,1.3531602863725503) circle (3.5pt);
\end{scriptsize}
\end{tikzpicture}
	}
	\caption{We remove the line $L$.}
	\label{fig: Lemma_4}
	\end{figure}
	Now, since we are looking at curves of degree $a'+1$, the lines $\overline{Q_1P_i}$, for $i = 2t_1+1,\ldots,s_1$, become fixed components and can be removed, see Figure \ref{fig: Lemma_5}, i.e.,
	$$
		\dim\mathcal{L}_{a'+1}(\bbJ) = \dim\mathcal{L}_{a''+1}(\bbJ'),
	$$
	with 
	$$
	\bbJ' = 
	\begin{cases}
		a''Q_1 + 2{Q}_2 +  {J}_1 + \ldots + {J}_{t_1}, & \text{ for } c = 0,1,2,3; \\
		a''Q_1 + {Q}_2 +  {J}_1 + \ldots + {J}_{t_1}, & \text{ for } c = 4; \\
	\end{cases}
	$$
	where
	$$
		a'' = a'-(s_1-2t_1) = 
		\begin{cases}
			2k - (3k-2k) = k& \text{for } c = 0; \\
			2k - (3k+1-2(k+1)) = k+1& \text{for } c = 1;\\
			(2k+1) - (3k+1- 2k) = k & \text{for } c = 2; \\
			(2k+1) - (3k+2 - 2(k+1)) = k+1 & \text{for } c = 3;\\
			(2k+1) - (3k+3 - 2(k+1)) = k & \text{for } c = 4. \\
		\end{cases}
	$$
	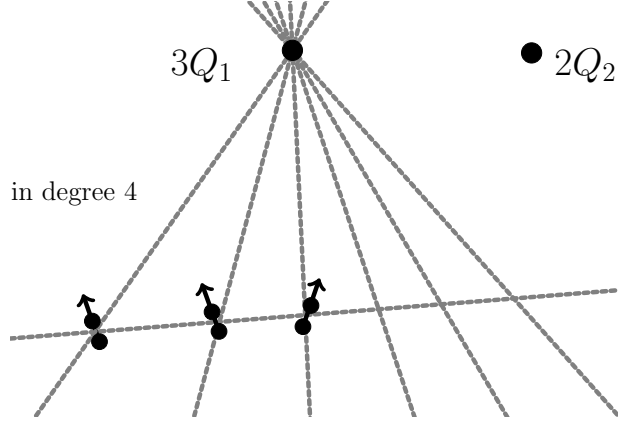
\begin{figure}[h]
	\centering
	\scalebox{0.85}{
	\begin{tikzpicture}[line cap=round,line join=round,x=1.0cm,y=1.0cm]
\clip(-3.,-2.5) rectangle (6.5,4.);
\draw [line width=2.pt,dotted,domain=-3.:6.5, color = gray] plot(\x,{(--6.606070669096441-2.653365106436809*\x)/0.88280517651701});
\draw [line width=2.pt,dotted,domain=-3.:6.5, color = gray] plot(\x,{(--9.676427200671663-2.92575777177952*\x)/1.7123646573334501});
\draw [line width=2.pt,dotted,domain=-3.:6.5, color = gray] plot(\x,{(--8.795388688536372-2.0466723518098617*\x)/1.8237980204281958});
\draw [line width=2.pt,dotted,domain=-3.:6.5, color = gray] plot(\x,{(--3.6452140416492202-6.69929374398375*\x)/-1.7971916286827576});
\draw [line width=2.pt,dotted,domain=-3.:6.5, color = gray] plot(\x,{(-0.3695083120389595--0.028408168922945176*\x)/0.3571005711237858});
\draw [->,line width=2.pt] (0.2648321501824796,-1.1618010130637089) -- (-0.007560515160232151,-0.4065304409771011);
\draw [line width=2.pt,dotted,domain=-3.:6.5, color = gray] plot(\x,{(--8.607111108960233-5.488725122958664*\x)/0.2637309371017562});
\draw [->,line width=2.pt] (1.5772695377428179,-1.0875121043338787) -- (1.8372807182972242,-0.33224153224727093);
\draw [->,line width=2.pt] (-1.6047720528515872,-1.3227603153116745) -- (-1.8647832334059937,-0.5798712280133718);
\draw [line width=2.pt,dotted,domain=-3.:6.5, color = gray] plot(\x,{(-4.489159349354937-5.241095427192564*\x)/-3.673581225579258});
\begin{scriptsize}
\draw [fill=black] (1.4125904789475021,3.2373534485241113) circle (4.5pt);
\draw[color=black] (0.,3.) node {\LARGE $3Q_1$};
\draw [fill=black] (5.153549103316712,3.1969296513445316) circle (4.5pt);
\draw[color=black] (6.,3.) node {\LARGE $2Q_2$};
\draw[color=black] (-2.,1.) node {\large in degree $4$};
\draw [fill=black] (1.5772695377428179,-1.0875121043338787) circle (3.5pt);
\draw [fill=black] (-1.716205415946333,-1.0008417108157432) circle (3.5pt);
\draw [fill=black] (-1.6047720528515872,-1.3227603153116745) circle (3.5pt);
\draw [fill=black] (0.2648321501824796,-1.1618010130637089) circle (3.5pt);
\draw [fill=black] (0.15570386993870922,-0.8567726333969978) circle (3.5pt);
\draw [fill=black] (1.7029018983959778,-0.7564612511753777) circle (3.5pt);
\end{scriptsize}
\end{tikzpicture}
}
\caption{We remove all the lines passing through $Q_1$ and containing the directions of the three $(2,2)$-jets.}
\label{fig: Lemma_5}
	\end{figure}
	
	Now, if $c = 0,1,2,3$, since we are looking at curves of degree $a''+1$, we have that the line $\overline{Q_1Q_2}$ is a fixed component and can be removed. After that, specialize the point $Q_2$ to a general point lying on $L$; see Figure \ref{fig: Lemma_6}.
	\begin{figure}[h]
		\begin{subfigure}[b]{0.45\textwidth}
		\scalebox{0.75}{\begin{tikzpicture}[line cap=round,line join=round,x=1.0cm,y=1.0cm]
\clip(-3.,-2.5) rectangle (7.,4.);
\draw [line width=2.pt,dotted,domain=-3.:7., color = gray] plot(\x,{(--3.6452140416492202-6.69929374398375*\x)/-1.7971916286827576});
\draw [line width=2.pt,dotted,domain=-3.:7., color = gray] plot(\x,{(-0.3695083120389595--0.028408168922945176*\x)/0.3571005711237858});
\draw [->,line width=2.pt] (0.2648321501824796,-1.1618010130637089) -- (-0.007560515160232151,-0.4065304409771011);
\draw [line width=2.pt,dotted,domain=-3.:7., color = gray] plot(\x,{(--8.607111108960233-5.488725122958664*\x)/0.2637309371017562});
\draw [->,line width=2.pt] (1.5772695377428179,-1.0875121043338787) -- (1.8372807182972242,-0.33224153224727093);
\draw [->,line width=2.pt] (-1.6047720528515872,-1.3227603153116745) -- (-1.8647832334059937,-0.5798712280133718);
\draw [line width=2.pt,dotted,domain=-3.:7., color = gray] plot(\x,{(-4.489159349354937-5.241095427192564*\x)/-3.673581225579258});
\draw [line width=2.pt,dotted,domain=-3.:7., color = gray] plot(\x,{(--12.167907574406456-0.04042379717957978*\x)/3.7409586243692097});
\begin{scriptsize}
\draw [fill=black] (1.4125904789475021,3.2373534485241113) circle (4.5pt);
\draw[color=black] (0.25,2.8) node {\LARGE $2Q_1$};
\draw [fill=black] (5.153549103316712,3.1969296513445316) circle (4.5pt);
\draw[color=black] (5.5,2.5) node {\LARGE $Q_2$};
\draw[color=black] (-2.,1.) node {\large in degree $3$};
\draw [fill=black] (1.5772695377428179,-1.0875121043338787) circle (3.5pt);
\draw [fill=black] (-1.716205415946333,-1.0008417108157432) circle (3.5pt);
\draw [fill=black] (-1.6047720528515872,-1.3227603153116745) circle (3.5pt);
\draw [fill=black] (0.2648321501824796,-1.1618010130637089) circle (3.5pt);
\draw [fill=black] (0.15570386993870922,-0.8567726333969978) circle (3.5pt);
\draw [fill=black] (1.7029018983959778,-0.7564612511753777) circle (3.5pt);
\end{scriptsize}
\end{tikzpicture}}
		\caption{}
		\end{subfigure}
		~~~
		\begin{subfigure}[b]{0.45\textwidth}
		\scalebox{0.75}{
		\begin{tikzpicture}[line cap=round,line join=round,x=1.0cm,y=1.0cm]
\clip(-3.5,-3.) rectangle (7.,4.5);
\draw [line width=2.pt,dotted,domain=-3.5:7., color = gray] plot(\x,{(--3.6452140416492202-6.69929374398375*\x)/-1.7971916286827576});
\draw [line width=2.pt,dotted,domain=-3.5:7., color = gray] plot(\x,{(-0.3695083120389595--0.028408168922945176*\x)/0.3571005711237858});
\draw [->,line width=2.pt] (0.2648321501824796,-1.1618010130637089) -- (-0.007560515160232151,-0.4065304409771011);
\draw [line width=2.pt,dotted,domain=-3.5:7., color = gray] plot(\x,{(--8.607111108960233-5.488725122958664*\x)/0.2637309371017562});
\draw [->,line width=2.pt] (1.5772695377428179,-1.0875121043338787) -- (1.8372807182972242,-0.33224153224727093);
\draw [->,line width=2.pt] (-1.6047720528515872,-1.3227603153116745) -- (-1.8647832334059937,-0.5798712280133718);
\draw [line width=2.pt,dotted,domain=-3.5:7., color = gray] plot(\x,{(-4.489159349354937-5.241095427192564*\x)/-3.673581225579258});
\draw [line width=2.pt,dotted,domain=-3.5:7., color = gray] plot(\x,{(--15.681208319673534-3.9398807354301373*\x)/3.12470370179319});
\begin{scriptsize}
\draw [fill=black] (1.4125904789475021,3.2373534485241113) circle (4.5pt);
\draw[color=black] (0.5,3) node {\LARGE $2Q_1$};
\draw [fill=black] (4.537294180740692,-0.7025272869060261) circle (4.5pt);
\draw[color=black] (5.,0.) node {\LARGE $Q_2$};
\draw[color=black] (-2.,1.) node {\large in degree $3$};
\draw [fill=black] (1.5772695377428179,-1.0875121043338787) circle (3.5pt);
\draw [fill=black] (-1.716205415946333,-1.0008417108157432) circle (3.5pt);
\draw [fill=black] (-1.6047720528515872,-1.3227603153116745) circle (3.5pt);
\draw [fill=black] (0.2648321501824796,-1.1618010130637089) circle (3.5pt);
\draw [fill=black] (0.15570386993870922,-0.8567726333969978) circle (3.5pt);
\draw [fill=black] (1.7029018983959778,-0.7564612511753777) circle (3.5pt);
\end{scriptsize}
\end{tikzpicture}
		}
		\caption{}
		\end{subfigure}
		\caption{In ({\sc A}) we have removed the line $\overline{Q_1Q_2}$ and then, in ({\sc B}), we specialize the point $Q_2$ to lie on $L$.}
		\label{fig: Lemma_6}
	\end{figure}
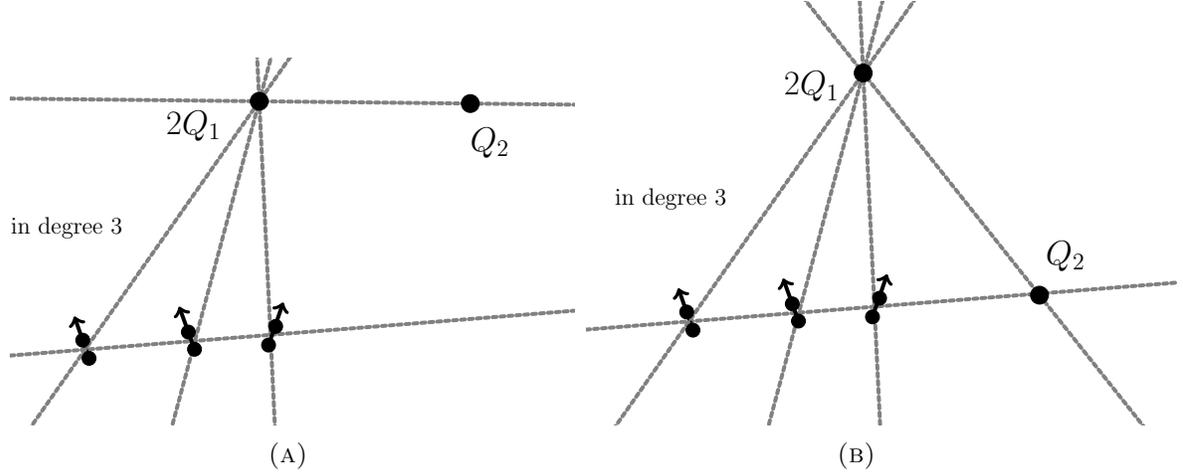
	
	Therefore, for any $c$, we reduced to computing the dimension of the linear system
	$
		\mathcal{L}_{\widetilde{a}+1}({\bbJ}''),
	$
	where ${\bbJ}'' = a'''Q_1 +  Q_2 + {J}_1 + \ldots + {J}_{t_1} $ with
	$$
		{a}''' = 
		\begin{cases}
			k-1 & \text{for } c = 0,2; \\
			k & \text{for } c = 1,3,4;
		\end{cases}
	$$
	and where
	\begin{itemize}
		\item $Q_1$ is a general point;
		\item $Q_2$ is a general point on $L$;
		\item $J_i$'s are $2$-jet points with support on $L$ and general direction.
	\end{itemize}
	Note that this can be done because, for $c \leq 2$, we may assume $k \geq 1$, since $a \geq 3$.
	
	Moreover,
	$$
		\deg({\bbJ}'' \cap L) = t_1 + 1 \quad = \quad
		\begin{cases}
		 	k+1 & \text{for } c = 0,2; \\
			k+3 & \text{for } c = 1,3,4.
		\end{cases} \quad = \quad {a}'''+2.
	$$
	Therefore, the line $L$ is a fixed component for 	$
		\mathcal{L}_{{a}'''+1}({\bbJ}'')
	$ and can be removed, i.e.,
	$$
		\dim\mathcal{L}_{{a}'''+1}({\bbJ}'') =
		\dim\mathcal{L}_{{a}'''}({a}'''Q_1 + \widetilde{P}_1 + \ldots + \widetilde{P}_{t_1}),
	$$
	where the $\widetilde{P}_i$'s are collinear. Since $t_1 = {a}'''+1$, this linear system is empty and this concludes the proof of the case $s = s_1$.
	
	\bigskip
	{\sc Step 3: case $s = s_2$.} First of all, note that the case $c = 4$ has been already proved since in that case we have $s_1 = s_2$. Moreover, the case $c = 1$ follows easily because we have that the linear system in the case $s = s_1$ has dimension $1$. 
	
	Hence, we are left just with the cases $c = 0,2,3$. We have
	$$
		s_2 = \begin{cases}
			3k+1 & \text{for } c = 0;\\
			3k+2 & \text{for } c = 2;\\
			3k+3 & \text{for } c = 3.\\
		\end{cases}, \quad \text{ and } \quad 
		a' = a-s_2 = \begin{cases}
			2k-1 & \text{for } c = 0;\\
			2k & \text{for } c = 2,3.
		\end{cases}, 
	$$ 
	Now, define
	$$
		t_2 = 
		\begin{cases}
		 	k & \text{for } c = 0;\\
			k+1 & \text{for } c = 2,3.
		\end{cases}
	$$
	Then, we proceed similarly as before. Note that $s_2 \geq 2t_2$. Now, we specialize $2t_2$ lines in such a way that, for $i = 1,\ldots,t_2$:
	\begin{itemize}
		\item the lines $\overline{Q_1P_{2i-1}}$ and $\overline{Q_1P_{2i}}$ both degenerate to a general line $R_i$ passing through $Q_1$;
		\item the point $P_{2i-1}$ and the point $P_{2i}$ both degenerate to a general point $\widetilde{P}_i$ on $R_i$.
	\end{itemize}
	In this way, from the degeneration of $J_{2i-1}$ and $J_{2i}$, we obtain the $(2,2)$-jet $W_i$ defined by the scheme-theoretic intersection $2R_i \cap 2\overline{P_{2i-1}P_{2i}}$, for any $i = 1,\ldots,t_2$. Moreover, we specialize the support of the $W_i$'s and the double point $2Q_2$ on a line $L$. Note that this can be done because $s_2 - 2t_2 \geq 0$. By abuse of notation, we call again ${\bbX}$ the scheme obtained after such a specialization. 
		
		In this way, we have that 
	$$
		\deg({\bbX}\cap L) = 
		\begin{cases}
			2k + 2 & \text{for } c = 0 \\
			2(k+1) + 2 & \text{for } c = 2,3; 
		\end{cases} = a'+3.
	$$
	Therefore, the line $L$ is a fixed component and can be removed. Hence,
	$$
		\dim\mathcal{L}_{a'+2}({\bbX}) = \dim\mathcal{L}_{a'+1}(\bbJ),
	$$
	where $\bbJ = a'Q_1 + Q_2 + {J}_1 + \ldots + {J}_{t_2} + J_{2t_2+1} + \ldots + J_{s_2}$, where
	\begin{itemize}
		\item ${J}_i$ is a $2$-jet with support on $\widetilde{P}_i\in L$ with general direction, for all $i = 1,\ldots, t_2$;
		\item ${J}_i$ is a $2$-jet contained in $\overline{Q_1P_i}$, for all $i = 2t_2+1,\ldots,s_2$.
	\end{itemize}
	Now, since we look at curves of degree $a'+1$, the lines $\overline{Q_1P_i}$, for $i = 2t_2+1,\ldots,s_2$, become fixed components and can be removed. Hence, we get
	$$
		\dim\mathcal{L}_{a'+1}(\bbJ) = \dim\mathcal{L}_{a''+1}(\bbJ'),
	$$
	where $\bbJ' = a''Q_1 + Q_2 + {J}_1 + \ldots + {J}_{t_2}$ and
	$$
		a'' = a' - (s_2 - 2t_2) = 
		\begin{cases}
			k-2 & \text{for } c = 0;\\
			k & \text{for } c = 2; \\
			k -1 & \text{for } c = 3.
		\end{cases}
	$$ 
	Hence,  
	$$
		\deg(\bbJ' \cap L) = 1+t_2 = 
		\begin{cases}
			k+1 & \text{for } c = 0;\\
			k+2 & \text{for } c = 2,3;
		\end{cases} > a''+1;
	$$
	the line $L$ is a fixed component and can be removed. Hence, we are left with the linear system $\mathcal{L}_{a''}(a''Q_1 + \widetilde{P}_1 + \ldots + \widetilde{P}_{t_2})$. Since $t_2 = a''+1$, this linear system is empty and this concludes the proof of the case $s = s_2$. 
\end{proof}

\section{Main result}\label{sec: main}
We are now ready to consider our general case. First, we answer to Question \ref{question: 0-dim P2}.

\begin{remark}[\sc Strategy of the proof]\label{rmk: strategy}
Our strategy to compute the dimension of $\mathcal{L}_{a+b}(\bbX_{a,b;s})$ goes as follows. Fix a general line $L$. Then, we consider a specialization ${\bbX} = aQ_1 + bQ_2 + {Y}_1 + \ldots + {Y}_s$ of the $0$-dimensional scheme $\bbX_{a,b;s}$, where the $Y_i$'s are $(3,2)$-points with support at $P_1,\ldots,P_s$, respectively, and such that:
\begin{itemize}
	\item ${Y}_1,\ldots,{Y}_x$ have support on $L$ and $\deg({Y}_i \cap L) = 3$, for all $i = 1,\ldots,x$;
	\item ${Y}_{x+1},\ldots,{Y}_{x+y}$ have support on $L$, $\deg({Y}_i \cap L) = 2$ and ${\rm Res}_L({Y}_i)$ is a $3$-jet contained in $L$, for all $i = x+1,\ldots,x+y$;
	\item ${Y}_{x+y+1},\ldots,{Y}_{x+y+z}$ have support on $L$, $\deg({Y}_i \cap L) = 2$ and ${\rm Res}_L({Y}_i)$ is a $2$-fat point, for all $i = x+y+1,\ldots,x+y+z$;
	\item ${Y}_{x+y+z+1},\ldots,{Y}_{s}$ are generic $(3,2)$-points. 
\end{itemize}
We will show that it is possible to chose $x,y,z$ in such a way that we start a procedure that allows us to remove twice the line $L$ and the line $\overline{Q_1Q_2}$ because, step-by-step, they are fixed components for the linear systems considered. A similar idea has been used by the authors, together with E. Carlini, to study Hilbert functions of triple points in $\bbP^1\times\bbP^1$; see \cite{CCO17}.

Finally, we remain with the linear system of curves of degree $a+b-2$ passing through the $0$-dimensional scheme $(a-2)Q_1+(b-2)Q_2 + {P}_{x+y+1}+\ldots+{P}_{x+y+z} + {Y}_{x+y+z+1} + \ldots +{Y}_{s}$, where the simple points ${P}_{x+y+1},\ldots,{P}_{x+y+z+1}$ are collinear and lie on the line $L$; see Figure \ref{fig: main}, where the degree goes down from $a+b$ to $a+b-2$ and finally to $a+b-4$. Hence, combining a technical lemma to deal with the collinear points (Lemma \ref{lemma: collinear}), we conclude our proof by a two-step induction, using the results of the previous section for $b \leq 2$.
\end{remark}

\begin{figure}[h]
\xymatrix{
		\includegraphics[scale=0.36]{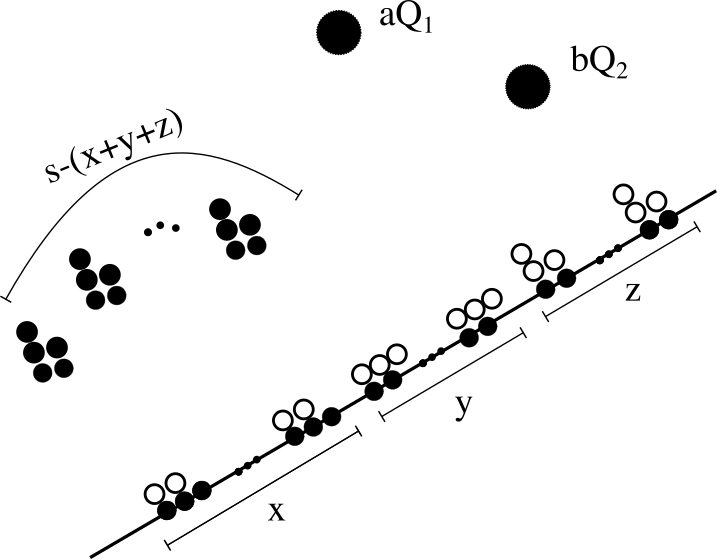}
		\ar@{=>}[r]<10ex> & \includegraphics[scale=0.36]{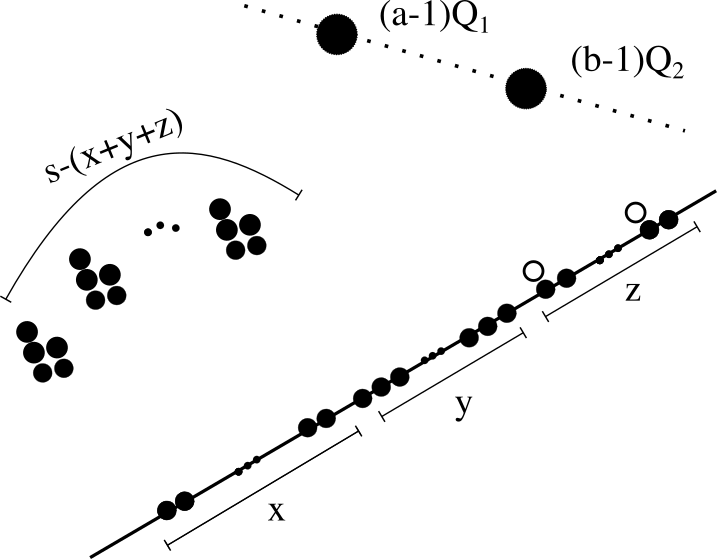} \ar@{=>}[d] \\
			& \includegraphics[scale=0.36]{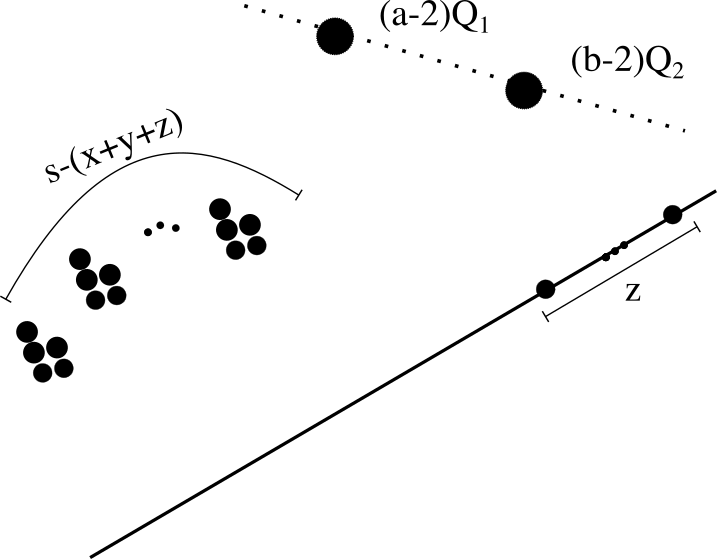}
		}
\caption{The three main steps of our proof of Theorem \ref{thm: main P2}.}
\label{fig: main}
\end{figure}

This strategy works in general, except for a few number of cases that we need to consider separately.

\begin{lemma}\label{lemma: small cases}
	In the same notation as above:
	\begin{enumerate}
		\item if $(a,b) = (3,3)$, then $\HF_{\bbX_{3,3;s}}(6) = \max\{0, 16-5s\}$;
		\item if $(a,b) = (5,3)$, then $\HF_{\bbX_{5,3;s}}(8) = \max\{0, 24-5s\}$;
		\item if $(a,b) = (4,4)$, then $\HF_{\bbX_{4,4;s}}(8) = \max\{0, 25-5s\}$.
	\end{enumerate}
\end{lemma}
\begin{proof}
	{\it (1)} If $(a,b) = (3,3)$, we have to consider $s_1 = 3$ and $s_2 = 4$. Let $\bbX_{3,3;3} = 3Q_1 + 3Q_2 + Y_1 + Y_2 + Y_3$, where $P_1,P_2,P_3$ is the support of $Y_1,Y_2,Y_3$, respectively. We specialize the scheme such that $\deg(Y_1 \cap \overline{Q_1P_1}) = 3$ and $\deg(Y_2\cap \overline{P_2P_3}) = \deg(Y_3\cap \overline{P_2P_3}) = 3$. Let $A$ be a generic point on the line $\overline{P_2P_3}$. Therefore, we have that the line $\overline{P_2P_3}$ is a fixed component for $\calL_{6}(\bbX_{3,3;3} + A)$ and we remove it. Let $\bbX' = {\rm Res}_L(\bbX_{3,3;3}+A)$. Then:
	\begin{itemize}
		\item the line $\overline{Q_1Q_2}$ is a fixed component for $\calL_5({\bbX}')$;
		\item the line $\overline{Q_1P_1}$ is a fixed component for $\calL_4({\rm Res}_{\overline{Q_1Q_2}}({\bbX'}))$;
		\item the line $\overline{P_2P_3}$ is a fixed component for $\calL_3({\rm Res}_{\overline{Q_1P_1}\cdot \overline{Q_1Q_2}}({\bbX}'))$.
	\end{itemize}
	Therefore, we obtain that
		$\dim\calL_6(\bbX_{3,3;3}+A) = \dim\calL_2(Q_1 + 2Q_2 + J_1)$, where $J_1$ is a $2$-jet lying on $\overline{Q_1P_1}$. The latter linear system is empty and therefore, since the expected dimension of $\calL_6(\bbX_{3,3;3})$ is $1$, we conclude that $\dim\calL_6(\bbX_{3,3;3}) = 1$. As a direct consequence, we also conclude that $\dim\calL_6(\bbX_{3,3;4}) = 0$.
	
	\medskip
	{\it (2)} If $(a,b) = (5,3)$, then $s_1 = 4$ and $s_2 = 5$. Let ${\bbX} = 5Q_1 + 3Q_2 + Y_1 + {Y}_2 + \ldots + {Y}_4$, where we specialize the support of ${Y}_2, Y_3, Y_4$ to be collinear on a line $L$ and $\deg(Y_2 \cap L) = \deg(Y_3 \cap L) = \deg(Y_4 \cap L) = 3$. 
	Now, $L$ is a fixed component for $\calL_8({\bbX})$ and $\overline{Q_1Q_2}$ is a fixed component for $ \calL_7({\rm Res}_L({\bbX}))$. Let $\bbX' = {\rm Res}_{\overline{Q_1Q_2}\cdot L}({\bbX}) = 4Q_1 + 2Q_2 + Y_1 + J_2 + J_3 + J_4$, where $J_i$'s are $2$-jets lying on $L$. Now, consider a generic point $A$ on $L$.
	In this way, $L$ is a fixed component for $\calL_6(\bbX' + A)$, $\overline{Q_1Q_2}$ is a fixed component for $\calL_5({\rm Res}_{L}({\bbX}'+A))$ and $\overline{Q_1P_1}$ is a fixed component for $\calL_4({\rm Res}_{\overline{Q_1Q_2}\cdot L}(\bbX'+A))$, where $P_1$ is the support of $Y_1$. Hence, we have 
	$$
		\dim\calL_6({\bbX}'+A) = \dim\calL_3(2Q_1 + Q_2 + 2P_1) = 3.
	$$
	Since the expected dimension of $\calL_6(\bbX')$ is $4$ and $\dim\calL_6({\bbX}'+A) = 3$, we conclude that $\dim\calL_{8}(\bbX_{5,3;4}) = \dim\calL_6({\bbX}') = 4$, as expected.
	
	In the case $s = s_2 = 5$, we consider ${\bbX} = 5Q_1 + 3Q_2 + Y_1 + Y_2 + {Y}_3 + \ldots + {Y}_5$, where the support of ${Y}_3, Y_4, Y_5$ are three collinear points ${P}_3, {P}_4$ and ${P}_5$, lying on a line $L$, and $\deg(Y_3 \cap L) = \deg(Y_4 \cap L) = \deg(Y_5 \cap L) = 3$. Then, $L$ is a fixed component for  $ \calL_8({\bbX})$ and $\overline{Q_1Q_2}$ is a fixed component for $ \calL_7({\rm Res}_L({\bbX}))$. Now, specializing the scheme such that $\deg(Y_1 \cap \overline{Q_1P_1}) = \deg(Y_2 \cap \overline{Q_1P_2}) = 3$, the lines $\overline{Q_1P_1}$ and $\overline{Q_1P_2}$ become also fixed components of the latter linear system and can be removed. 
Hence, 
	$$
	\dim \calL_8({\bbX}) = \dim \calL_4(\bbY),
	$$
	where $\bbY = 2Q_1 + 2Q_2 + J_1 + J_2 + {J}_3 + \ldots + {J}_5$, where $J_1,J_2$ are $2$-jets lying on $\overline{Q_1P_1}$ and $\overline{Q_1P_2}$, respectively, and $J_3,J_4,J_5$ are $2$-jets lying on $L$. Now:
	\begin{itemize}
		\item $L$ is a fixed component for $\calL_4(\bbY)$;
		\item $\overline{Q_1Q_2}$ is a fixed component for $\calL_3({\rm Res}_L(\bbY))$;
		\item $\overline{Q_1P_1}$ and $\overline{Q_1P_2}$ are fixed components for $\calL_2({\rm Res}_{\overline{Q_1Q_2}\cdot L}(\bbY))$.
	\end{itemize}
	From this, we conclude that $\dim\calL_8(\bbX) = \dim\calL_4(\bbY) = \dim\calL_0(Q_2) = 0$, as expected.
	
	\medskip
	{\it (3)} If $(a,b) = (4,4)$, we have $s_1 = s_2 = 5$. Consider ${\bbX} = 4Q_1 + 4Q_2 + Y_1+ Y_2+Y_3 + {Y}_4 + {Y}_5$, where the supports of ${Y}_4$ and ${Y}_5$ are collinear with $Q_2$ on a line $L$. In this specialization, we also assume that:
	\begin{itemize}
		\item $\deg({Y}_4 \cap L) = 3$;
		\item $\deg({Y}_5 \cap L) = 2$ and $\deg({\rm Res}_L({Y}_5) \cap L) = 3$.
	\end{itemize}
	Therefore, we obtain that $L$ is a fixed component for the linear system and can be removed twice. Hence,
	$$
		\dim\calL_8({\bbX}) = \dim\calL_6({\rm Res}_{2L}({\bbX})),
	$$
	where ${\rm Res}_{2L}({\bbX}) = \bbX_{4,2;3}$. By Lemma \ref{lemma: b = 2}, the latter linear system is empty and we conclude.
\end{proof}

The following lemma is a well-known tool to study the Hilbert function of $0$-dimensional schemes which have some reduced component lying on a line. 

\begin{lemma}\label{lemma: collinear}{\rm \cite[Lemma 2.2]{CGG05-SegreVeronese}}
 Let $\bbX \subset \bbP^2$ be a $0$-dimensional scheme, and let $P_1,\ldots,P_s$ be general points on a line $L$. Then:
\begin{enumerate}
\item if $ \dim\calL_{d}(\bbX+P_1+\ldots+P_{s-1}) > \dim\calL_{d-1}({\rm Res}_L(\bbX))$, then
$$
  \dim\calL_d(\bbX+P_1+\ldots+P_s) = \dim\calL_d(\bbX) - s;$$
\item 
if $ \dim\calL_{d-1}({\rm Res}_L(\bbX))=0$ and $\dim\calL_{d}(\bbX)\leq s$, then
$\dim\calL_{d}(\bbX+P_1+\ldots+P_{s})= 0$.
\end{enumerate}
\end{lemma}

Now, we are ready to prove our first main result.

\begin{theorem}\label{thm: main P2}
	Let $a,b$ be positive integers with $ab > 1$. Then, let $\bbX_{a,b;s} \subset \bbP^2$ as in the previous section. Then,
	$$
		\dim \mathcal{L}_{a+b}(\bbX_{a,b;s}) = \max\{0, (a+1)(b+1) - 5s\}.
	$$
\end{theorem}
\begin{proof}
	We assume that $a \geq b \geq 3$, since the cases $b = 1$ and $b = 2$ are treated in Lemma \ref{lemma: b = 1}, Lemma \ref{lemma: a = 2} and Lemma \ref{lemma: b = 2}. Moreover, as explained in Section \ref{sec: super- and sub-abundance}, we consider $s \leq s_2$.
	
	First, we show how to choose the numbers $x,y,z$ described in the Remark \ref{rmk: strategy}.
	Let $a+b = 5h+c$, with $0\leq c \leq 4$. Then, we fix
	\begin{center}
	\begin{tabular}{c || c | c | c}
		c & x & y & z \\
		\hline
		0 & h+1 & h-1 & 0 \\ 
		1 & h & h-2 & 3 \\
		2 & h+1 & h-1 & 1 \\
		3 & h & h-2 & 4 \\
		4 & h+1 & h-1 & 2 \\
	\end{tabular}
	\end{center}
	Note that in order to make these choices, we need to assume that $y \geq 0$ for any $c$. In particular, this means that the cases with $a+b$ equal to $6$ and $8$ have to be treated in a different way. Hence, the only cases we have to treat differently from the main strategy are $(a,b) = (3,3), (5,3), (4,4)$, already considered in Lemma \ref{lemma: small cases}. 
	
	Note that, for any $h,c$, we have that $x+y+z \leq s_1 = \left\lfloor \frac{(a+1)(b+1)}{5} \right\rfloor$. Indeed, it is enough to see that 
	$$
		\frac{(a+1)(b+1)}{5} - (x+y+z) \geq 0.
	$$
	By direct computation, we have that
	$$
		(a+1)(b+1) - 5(x+y+z) = (a-1)(b-1) - 
		\begin{cases}
			0 & \text{ for } c = 0; \\
			3 & \text{ for } c = 1; \\
			1 & \text{ for } c = 2; \\
			4 & \text{ for } c = 3; \\
			2 & \text{ for } c = 4.
		\end{cases}
	$$
	Since $a,b \geq 3$, we conclude that $x+y+z \leq s_1$.

	With these assumptions, by direct computation, we obtain that, for any $c$,
	$$
		\deg({\bbX} \cap L) = 3x+2y+2z = a+b+1 \quad \text{ and } \quad 
		\deg({\rm Res}_L({\bbX}) \cap L) = 2x+3y+2z = a+b-1.
	$$
	Therefore, we can do the following procedure (see Figure \ref{fig: main}):
	\begin{itemize}
		\item the line $L$ is a fixed component for $\calL_{a+b}({\bbX})$ and it can be removed;
		\item the line $\overline{Q_1Q_2}$ is a fixed component for $\calL_{a+b-1}({\rm Res}_L({\bbX}))$ and it can be removed;
		\item the line $L$ is a fixed component for $\calL_{a+b-2}({\rm Res}_{\overline{Q_1Q_2}\cdot L}({\bbX}))$ and it can be removed;
		\item the line $\overline{Q_1Q_2}$ is a fixed component for $\calL_{a+b-3}({\rm Res}_{\overline{Q_1Q_2}\cdot 2L}({\bbX}))$ and it can be removed.
	\end{itemize}
	Denote $\bbY = {\rm Res}_{2\overline{Q_1Q_2}\cdot 2L}({\bbX})$ which is the union of $\bbX_{a-2,b-2;s-(x+y+z)}$ and a set of $z$ general collinear points on $L$. By the previous reductions, we have that
	$$
		\dim\calL_{a+b}({\bbX}) = \dim\calL_{a+b-4}(\bbY).
	$$
	Now, case by case, we prove that the latter linear system has the expected dimension. In these computations, we proceed by induction on $b$ (with base cases $b = 1$ and $b = 2$ proved in Lemma \ref{lemma: b = 1}, Lemma \ref{lemma: a = 2} and Lemma \ref{lemma: b = 2}) to deal with $\bbX_{a-2,b-2;s-(x+y+z)}$; and by using Lemma \ref{lemma: collinear} to deal with the the $z$ general collinear points that we denote by $A_1,\ldots,A_z$.
	
	\medskip
	\noindent {\sc Case $c = 0$.} In this case, since $z = 0$, we conclude just by induction on $b$; indeed
	\begin{align*}
		\dim\calL_{a+b-4}(\bbY) & = \dim\calL_{a+b-4}(\bbX_{a-2,b-2;s-2h}) = 
		\max\{0, (a-1)(b-1) - 5(s-2h) \} = \\ 
		& = \max\{0, (a+1)(b+1) - 5s \} = {\rm exp}.\dim\calL_{a+b}(\bbX_{a,b;s}).
	\end{align*}
	
	\medskip
	\noindent {\sc Case $c = 1$.} We first consider the case $s = s_1$. We want to use Lemma \ref{lemma: collinear}, so, since $z = 3$, we compute the following:
	\begin{align*}
		\dim\calL_{a+b-4}&(\bbX_{a-2,b-2;s_1-(2h+1)} + A_1 + A_2) \\ 
		& = \dim\calL_{a+b-4}(\bbX_{a-2,b-2;s_1-(2h+1)}) - 2 & \hfill \text{(since 2 points are always general)}\\
		& = \max\{0,(a-1)(b-1)-5(s_1-(2h+1))\}-2 & \hfill \text{(by induction)} \\
		 & =  \max\{0, (a+1)(b+1)-5s_1+1\} = \\
		 & = (a+1)(b+1)-5s_1+1.
	\end{align*}
	\begin{align*}
		\dim \calL_{a+b-5}&({\rm Res}_L(\bbY)) = \dim\calL_{a+b-5}(\bbX_{a-2,b-2;s_1-(2h+1)}) = \\
		& = \dim\calL_{a+b-6}(\bbX_{a-3,b-3;s_1-(2h+1)}) = 
		& \hfill \text{(since $\overline{Q_1Q_2}$ is a fixed component)} \\ 
		& = \max\{0,(a-2)(b-2)-5(s_1-(2h+1))\}  & \hfill \text{(by induction)} \\
		& = \max\{0, ab+7-5s_1\}.
	\end{align*}
	Now, since
	\begin{align*}
		(a+1)(b+1) - 5s_1+ 1 > 0, \quad \quad \text{ and } \\
		\left( (a+1)(b+1) - 5s_1 + 1\right) - (ab+7-5s_1) = a+b-5 > 0,
	\end{align*}
	by Lemma \ref{lemma: collinear}(1), we have that 
	\begin{align*}
		\dim\calL_{a+b-4}(\bbY) & = \max\{0, (a-1)(b-1)-5(s_1-(2h+1)) - 3\} = {\rm exp}.\dim\calL_{a+b}(\bbX_{a,b;s_1}).
	\end{align*}
	Now, consider $s = s_2$. Then, we have
	\begin{align*}
		\dim \calL_{a+b-5}&({\rm Res}_L(\bbY)) = \dim\calL_{a+b-5}(\bbX_{a-2,b-2;s_2-(2h+1)}) \\
		& = \dim\calL_{a+b-6}(\bbX_{a-3,b-3;s_2-(2h+1)}) & \hfill \text{(since $\overline{Q_1Q_2}$ is a fixed component)} \\ 
		& = \max\{0,(a-2)(b-2)-5(s_2-(2h+1))\} & \hfill \text{(by induction)}  \\
		& = \max\{0, ab+7-5s_2\} = 0,
	\end{align*}
	where the latter equality is justified by the fact that, by definition of $s_2$ (see Section \ref{sec: super- and sub-abundance}), we have
	$$
	ab+7-5s_2 \leq ab + 7 - (a+1)(b+1) = 6 - (a+b) \leq 0.
	$$
	Moreover, by definition of $s_2$,
	\begin{align*}
		\dim \calL_{a+b-4}({\rm Res}_L(\bbY)) & = \dim\calL_{a+b-4}(\bbX_{a-2,b-2;s_2-(2h+1)}) = \\
		& = \max\{0, (a-1)(b-1)-5(s_2-(2h+1))\} = & \hfill \text{(by induction)} \\
		& = \max\{0, (a+1)(b+1)-5s_2+3\} \leq 3.
	\end{align*}	
	Hence, we can apply Lemma \ref{lemma: collinear}(2) and conclude that, for $s = s_2$,
	$$\dim\calL_{a+b-4}(\bbY) = 0 = {\rm exp}.\dim\calL_{a+b}(\bbX_{a,b;s_2}).$$
	
	\medskip
	\noindent {\sc Case $c = 2$ and $c = 4$.} Since a generic simple point or two generic simple points always impose independent conditions on a linear system of curves, we conclude this case directly by induction on $b$. Indeed, for $c = 2$, we have
	\begin{align*}
		\dim\calL_{a+b-4}(\bbY) & = \dim\calL_{a+b-4}(\bbX_{a-2,b-2;s-(2h+1)}) - 1 = \\
		& = \max\{0, (a-1)(b-1)-5(s-(2h+1)) \} - 1 = & \hfill \text{(by induction)} \\
		& = \max\{0, (a+1)(b+1)-5s\} = {\rm exp}.\dim\calL_{a+b}(\bbX_{a,b;s}).
	\end{align*}
	Similarly, for $c = 4$, we have
	\begin{align*}
		\dim\calL_{a+b-4}(\bbY) & = \dim\calL_{a+b-4}(\bbX_{a-2,b-2;s-(2h+2)}) - 2 = \\
		& = \max\{0, (a-1)(b-1)-5(s-(2h+2))\} - 2 = & \hfill \text{(by induction)} \\
		& = \max\{0, (a+1)(b+1)-5s\} = {\rm exp}.\dim\calL_{a+b}(\bbX_{a,b;s}).
	\end{align*}
	
	\medskip
	\noindent {\sc Case $c = 3$.} We first consider the case $s = s_1$. We want to use Lemma \ref{lemma: collinear}, so, since $z = 4$, we compute the following:
	\begin{align*}
		\dim\calL_{a+b-4}& (\bbX_{a-2,b-2;s_1-(2h+2)} + A_1 + A_2 + A_3) \\
		& \geq \dim\calL_{a+b-4}(\bbX_{a-2,b-2;s_1-(2h+2)}) - 3  \\
		& = \max\{0, (a-1)(b-1)-5(s_1-(2h+2)) \} - 3 & \hfill \text{(by induction)} \\
		 & =  \max\{0, (a+1)(b+1)-5s_1+1\} = \\
		 & = (a+1)(b+1)-5s_1+1 \geq 1.
	\end{align*}
	\begin{align*}
		\dim \calL_{a+b-5}&({\rm Res}_L(\bbY)) = \dim\calL_{a+b-5}(\bbX_{a-2,b-2;s_1-(2h+2)}) = \\
		& =  \dim\calL_{a+b-6}(\bbX_{a-3,b-3;s_1-(2h+2)}) & \hfill \text{(since $\overline{Q_1Q_2}$ is a fixed component)} \\ 
		& = \max\{0,(a-2)(b-2)-5(s_1-(2h+2))\} & \hfill \text{(by induction)} \\
		& = \max\{0, ab+8-5s_1\}.
	\end{align*}
	Since
	\begin{align*}
		(a+1)(b+1) - 5s_1 + 1 > 0, \quad \quad \text{ and } \\
		\left( (a+1)(b+1) - 5s_1 + 1\right) - (ab+8-5s_1) = a+b-6 > 0,
	\end{align*}
	by Lemma \ref{lemma: collinear}(1), we have that 
	\begin{align*}
		\dim\calL_{a+b-4}(\bbY) & = \max\{0, (a-1)(b-1)-5(s_1-(2h+2))\} - 4 = \\
		& = (a+1)(b+1) - 5s_1 = {\rm exp}.\dim\calL_{a+b}(\bbX_{a,b;s_1}).
	\end{align*}
	Now, consider $s = s_2$. We have
	\begin{align*}
		\dim \calL_{a+b-5}&({\rm Res}_L(\bbY)) = \dim\calL_{a+b-5}(\bbX_{a-2,b-2;s_2-(2h+2)}) = \\
		& =  \dim\calL_{a+b-6}(\bbX_{a-3,b-3;s_2-(2h+2)}) & \hfill \text{(since $\overline{Q_1Q_2}$ is a fixed component)} \\ 
		& = \max\{0,(a-2)(b-2)-5(s_2-(2h+2))\} & \hfill \text{(by induction)}  \\
		& = \max\{0, ab+8-5s_2\} = 0,
	\end{align*}
	where the latter equality is justified by the fact that, by definition of $s_2$ (see Section \ref{sec: super- and sub-abundance}), we have
	$$
		ab+8-5s_2 \leq ab+8-(a+1)(b+1) = 7 - (a+b) \leq 0.	
	$$
	Moreover, by definition of $s_2$,
	\begin{align*}
		\dim \calL_{a+b-4}({\rm Res}_L(\bbY)) & = \dim\calL_{a+b-4}(\bbX_{a-2,b-2;s_2-(2h+2)}) = \\
		& = \max\{0, (a-1)(b-1)-5(s_2-(2h+2))\} = & \hfill \text{(by induction)}  \\
		& = \max\{0, (a+1)(b+1)-5s_2+4\} \leq 4.
	\end{align*}	
	Hence, we can apply Lemma \ref{lemma: collinear}(2) and conclude that, for $s = s_2$,
	$$\dim\calL_{a+b-4}(\bbY) = 0 = {\rm exp}.\dim\calL_{a+b}(\bbX_{a,b;s_2}).$$
\end{proof}

By multiprojective-affine-projective method, from Theorem \ref{thm: main P2}, we answer to Question \ref{question: 0-dim P1xP1}.
\begin{corollary}\label{corollary: main P1xP1}
	Let $a, b$ be integers with $ab > 1$. Let $\bbX\subset\bbP^1\times\bbP^1$ be the union of $s$ many $(3,2)$-points with generic support and generic direction. Then,
	$$
		\HF_\bbX(a,b) = \min\{(a+1)(b+1), 5s\}.
	$$
\end{corollary}

In conclusion, by the relation between the Hilbert function of schemes of $(3,2)$-points in $\bbP^1 \times \bbP^1$ and the dimension of secant varieties of the tangential variety of Segre-Veronese surfaces (see Section \ref{section: apolarity}), we can prove our final result.

\begin{theorem}\label{thm: main}
	Let $a,b$ be positive integers with $ab>1$. Then, the tangential variety of any Segre-Veronese surface $SV_{a,b}$ is non defective, i.e., all the secant varieties have the expected dimension.
\end{theorem}

\bibliographystyle{alpha}
\bibliography{references}

\begin{thebibliography}{BCGI09}

\bibitem[AB13]{AB13-SegreVeronese}
H.~Abo and M.~C. Brambilla.
\newblock On the dimensions of secant varieties of segre-veronese varieties.
\newblock {\em Annali di Matematica Pura ed Applicata}, 192(1):61--92, 2013.

\bibitem[Abr08]{Abr08}
S.~Abrescia.
\newblock About defectivity of certain segre-veronese varieties.
\newblock {\em Canad. J. Math}, 60(5):961--974, 2008.

\bibitem[AH92a]{AHb}
J.~Alexander and A.~Hirschowitz.
\newblock La m{\'e}thode d'{H}orace {\'e}clat{\'e}e: application {\`a}
  l'interpolation en degr{\'e}e quatre.
\newblock {\em Inventiones mathematicae}, 107(1):585--602, 1992.

\bibitem[AH92b]{AHa}
J.~Alexander and A.~Hirschowitz.
\newblock Un lemme d'{H}orace diff{\'e}rentiel: application aux
  singularit{\'e}s hyperquartiques de $\mathbb{P}^5$.
\newblock {\em J. Algebraic. Geom.}, 1(3):411--426, 1992.

\bibitem[AH95]{AH}
J.~Alexander and A.~Hirschowitz.
\newblock Polynomial interpolation in several variables.
\newblock {\em Journal of Algebraic Geometry}, 4(4):201--222, 1995.

\bibitem[AH00]{AH00}
J.~Alexander and A.~Hirschowitz.
\newblock An asymptotic vanishing theorem for generic unions of multiple
  points.
\newblock {\em Inventiones Mathematicae}, 140(2):303--325, 2000.

\bibitem[AOP09]{AOP}
H~Abo, G.~Ottaviani, and C.~Peterson.
\newblock Induction for secant varieties of segre varieties.
\newblock {\em Transactions of the American Mathematical Society},
  361(2):767--792, 2009.

\bibitem[AV18]{AV18}
H.~Abo and N.~Vannieuwenhoven.
\newblock Most secant varieties of tangential varieties to veronese varieties
  are nondefective.
\newblock {\em Transactions of the American Mathematical Society},
  370(1):393--420, 2018.

\bibitem[BCGI09]{BCGI09}
A.~Bernardi, M.~V. Catalisano, A.~Gimigliano, and M.~Id{\`a}.
\newblock Secant varieties to osculating varieties of veronese embeddings of
  pn.
\newblock {\em Journal of Algebra}, 321(3):982--1004, 2009.

\bibitem[CCO17]{CCO17}
Enrico Carlini, Maria~Virginia Catalisano, and Alessandro Oneto.
\newblock On the hilbert function of general fat points in
  $\bbp^1\times\bbp^1$.
\newblock {\em arXiv preprint arXiv:1711.06193}, 2017.

\bibitem[CGG02]{CGG02-Segre}
M.~V. Catalisano, A.~V. Geramita, and A.~Gimigliano.
\newblock Ranks of tensors, secant varieties of segre varieties and fat points.
\newblock {\em Linear algebra and its applications}, 355(1-3):263--285, 2002.

\bibitem[CGG05]{CGG05-SegreVeronese}
M.~V. Catalisano, A.~V. Geramita, and A~Gimigliano.
\newblock Higher secant varieties of {S}egre-{V}eronese varieties.
\newblock {\em Projective varieties with unexpected properties}, pages 81--107,
  2005.

\bibitem[CGG11]{CGG11-SegreP1}
M.~V. Catalisano, A.~V. Geramita, and A.~Gimigliano.
\newblock Secant varieties of $\mathbb{P}^1\times\ldots\times\mathbb{P}^1$
  (n-times) are {NOT} defective for $n\geq 5$.
\newblock {\em Journal of Algebraic Geometry}, 20:295--327, 2011.

\bibitem[Ger96]{G}
A.~V. Geramita.
\newblock Inverse systems of fat points: {W}aring's problem, secant varieties
  of {V}eronese varieties and parameter spaces for {G}orenstein ideals.
\newblock In {\em The Curves Seminar at Queen's}, volume~10, pages 2--114,
  1996.

\bibitem[IK99]{IK}
A.~Iarrobino and V.~Kanev.
\newblock {\em Power sums, Gorenstein algebras, and determinantal loci}.
\newblock Springer, 1999.

\bibitem[Lan12]{L}
J.~M. Landsberg.
\newblock {\em Tensors: Geometry and applications}, volume 128.
\newblock American Mathematical Soc., 2012.

\bibitem[Ter11]{T11}
A.~Terracini.
\newblock Sulle $v_k$ per cui la variet{\`a} degli $s_h$ (h+1) seganti ha
  dimensione minore dell'ordinario.
\newblock {\em Rendiconti del Circolo Matematico di Palermo (1884-1940)},
  31(1):392--396, 1911.

\end{thebibliography}
\end{document}